\documentclass[11pt]{amsart}

\usepackage{geometry}               
\usepackage[utf8]{inputenc}        
\usepackage[english]{babel}        
\addto\shorthandsspanish{\spanishdeactivate{<>}} 
\usepackage[T1]{fontenc}           

\usepackage{amssymb}               
\usepackage{amsmath}               
\usepackage{calligra, mathrsfs}
\usepackage{lmodern}
\numberwithin{equation}{section}
\usepackage{amsthm}                
\usepackage{stmaryrd}
\SetSymbolFont{stmry}{bold}{U}{stmry}{m}{n}
\usepackage{mathtools}
\usepackage[inline, shortlabels]{enumitem}
\usepackage[mathcal]{euscript}     
\usepackage{tikz}                  
\usepackage{tikz-cd}               
\usetikzlibrary{babel}
\usepackage{latexsym}              
\usepackage{pgfplots}              
\pgfplotsset{compat=1.16}
\usepackage{subfig}                
\usepackage{array}
\usepackage{adjustbox}

\usepackage{url}
\usepackage{xcolor}
\definecolor{darkgreen}{RGB}{20,160,20}
\definecolor{myblue}{RGB}{120, 94, 240}
\definecolor{myorange}{RGB}{254, 97, 0}
\definecolor{mypink}{RGB}{220, 38, 127}
\usepackage[pdftex]{hyperref}
\hypersetup{
    colorlinks=true,
    linkcolor=blue,
    citecolor=darkgreen,      
    urlcolor=black,
    linktocpage=true
}

\usepackage{import}
\usepackage{xifthen}
\usepackage{pdfpages}
\usepackage{transparent}

\makeindex

\counterwithout{equation}{section}

\theoremstyle{plain}
\newtheorem{theorem}{Theorem}[section]
\newtheorem{lemma}[theorem]{Lemma}
\newtheorem{corollary}[theorem]{Corollary}
\newtheorem{proposition}[theorem]{Proposition}

\theoremstyle{definition}
\newtheorem{definition}[theorem]{Definition}
\newtheorem{example}[theorem]{Example}
\newtheorem{notation}[theorem]{Notation}
\newtheorem{remark}[theorem]{Remark}
\newtheorem{conjecture}[theorem]{Conjecture}

\newcommand{\B}{\mathcal{B}}
\newcommand{\C}{\mathbb{C}}

\newcommand{\N}{\mathbb{N}}

\renewcommand{\P}{\mathbb{P}}

\newcommand{\R}{\mathbb{R}}
\newcommand{\X}{\mathcal{X}}
\newcommand{\Z}{\mathbb{Z}}
\newcommand{\frakZ}{\mathfrak{Z}}
\renewcommand{\SS}{\mathbb{S}}
\newcommand{\DD}{\mathbb{D}}
\newcommand{\BB}{\mathbb{B}}

\renewcommand{\i}{\omega}

\newcommand{\hr}{\hat{r}}
\newcommand{\hkappa}{\hat{\kappa}}

\newcommand{\bmu}{\boldsymbol{\mu}}
\newcommand{\id}{\mathrm{id}}
\newcommand{\dom}{{\mathrm{dom}}}

\newcommand{\lambdastd}{\lambda_\mathrm{std}}

\newcommand{\CZ}{\mathrm{CZ}}

\newcommand{\HF}{\mathrm{HF}}
\newcommand{\dphi}{\widetilde{\varphi^m}}

\newcommand{\F}{M}

\renewcommand{\log}{\mathrm{log}}

\DeclareMathOperator{\ac}{ac}

\DeclareMathOperator{\ord}{ord}
\DeclareMathOperator{\mult}{mult}

\DeclareMathOperator{\Spec}{Spec}

\DeclareMathOperator{\sheafhom}{\mathscr{H}\text{\kern -3pt {\calligra\large om}}\,}

\DeclareMathOperator{\Fix}{\mathrm{Fix}}

\makeatletter
\DeclareFontEncoding{LS1}{}{}
\DeclareFontSubstitution{LS1}{stix}{m}{n}
\DeclareMathAlphabet{\mathcal}{LS1}{stixscr}{m}{n}
\makeatother

\title{The Arc-Floer conjecture for plane curves}
\author{Javier de la Bodega, Eduardo de Lorenzo Poza}

\begin{document}

\begin{abstract}
    In \cite{budur2020cohomology} it was conjectured that the compactly supported cohomology of the restricted $m$-contact locus of an isolated hypersurface singularity coincides, up to a shift, with the Floer cohomology of the $m$-th iterate of the monodromy of the Milnor fiber. In this paper we give an affirmative answer to this conjecture in the case of plane curves.
\end{abstract}

\maketitle

\section{Introduction}

Let $f \in \C\{z_0,\dots,z_n\}$ be a convergent power series that defines an isolated hypersurface singularity at the origin. The celebrated Monodromy Conjecture \cite[\S 2.4]{denef1998} aims to establish a connection between the \emph{contact loci} of $f$ (of algebraic nature) and the \emph{Milnor fiber}  of $f$ (of topological nature). In this paper we follow the philosophy of \cite{budur2020cohomology} to strengthen the connection between these two invariants from a different point of view.

For every integer $m \geq 1$, the \emph{restricted $m$-contact locus} of $f$, denoted by $\X_m$, is the set of $m$-jets in $\C^{n+1}$ centered at $0$ with intersection multiplicity $m$ with $f$ and angular component 1, see Definition \ref{contact-locus-def}. Since $\X_m$ is defined by polynomial equations, it is a (possibly singular) complex affine algebraic variety.

On the other hand, it is a classical result by Milnor \cite[Ch. 5]{milnor1968} that there exist $0 < \delta \ll \varepsilon \ll 1$ such that
\[
f: f^{-1}(\partial \DD_\delta) \cap \BB_\epsilon \longrightarrow \partial \DD_\delta
\]
is a $C^\infty$-locally trivial fibration, called the \emph{Milnor fibration}. The fiber of this fibration is known as the \emph{Milnor fiber} of $f$, and we denote it by $\F_f$, or simply by $\F$ when the function $f$ is clear from the context. A key observation is that $\F$ has a natural symplectic manifold structure. More precisely, for $\varepsilon > 0$ small enough $\F$ is a Liouville domain when endowed with the restriction of the 1-form
\[
\lambdastd = \frac{1}{2}\sum_{i=0}^n (x_idy_i-y_idx_i)
\]
of $\C^{n+1}$, where $z_i=x_i+\imath y_i, \imath = \sqrt{-1}$. Moreover, it admits a compactly supported exact monodromy $\varphi: \F \to \F$ for the Milnor fibration, see \cite[\S 5]{bobadilla2022}, \cite[\S 3]{mclean2019}. It is in this setting that Seidel \cite[\S 4]{seidel2001}, cf. \cite[\S 4]{mclean2019}, introduced a cohomology theory known as the \emph{Floer cohomology of $\varphi$}, denoted by $\mathrm{HF}^\bullet(\varphi,+)$. This was an adaptation of Dostoglou and Salamon's Floer homology for symplectomorphisms on simply connected monotone symplectic manifolds \cite[\S 2]{dostoglou-salamon94} to exact symplectomorphisms on Liouville domains.

In \cite{budur2020cohomology}, Budur, Fernández de Bobadilla, Lê and Nguyen conjectured that the cohomology with compact support of $\X_m$ (considered with its analytic topology) coincides, up to a shift, with the Floer cohomology of $\varphi^m$. The conjecture was made based on the following evidence:
\begin{enumerate}
    \item As explained in \cite[\S 4]{mclean2019}, the Euler characteristic of $\mathrm{HF}^\bullet(\varphi^m,+)$ is equal to $(-1)^n\Lambda_{\varphi^m}$, where $\Lambda_{\varphi^m}$ denotes the Lefschetz number of $\varphi^m$. Denef and Loeser proved in \cite{denef-loeser2002} that the compactly supported Euler characteristic of $\X_m$ is equal to $\Lambda_{\varphi^m}$. Thus the conjecture may be seen as a lift of this result about Euler characteristics to a statement about cohomology.
    \item Using an $m$-separating log resolution, McLean \cite{mclean2019} constructed a spectral sequence converging to $\mathrm{HF}^\bullet(\varphi^m,+)$. In turn, Budur et al. \cite{budur2020cohomology} constructed an analogous spectral sequence converging to $H_c^\bullet(\X_m)$.
    \item Also in \cite{budur2020cohomology}, Budur et al. showed the conjecture holds when $m$ equals the multiplicity of $f$.
\end{enumerate}

However, there was no class of singularities for which the conjecture was known to hold. The main purpose of this paper is to provide the first such example. Namely, we show the conjecture holds for $n=1$, i.e. for the case of plane curves:

\begin{theorem} \label{main-result}
    Let $f \in \C\{x,y\}$ a reduced convergent power series such that $f(0,0)=0$. For every integer $m \geq 1$ there is an isomorphism
    \[
    H^{\bullet+2m+1}_c(\X_m) \cong\ \HF_{\bullet}(\varphi^m,+).
    \]
\end{theorem}

As the reader may have noticed, the statement of Theorem \ref{main-result} involves a homological version of Seidel's Floer cohomology. Indeed, Uljarevic \cite[\S 2]{uljarevic2014} introduced the \emph{Floer homology} of $\varphi$, which was improved by Fernández de Bobadilla and Pe{\l}ka \cite[\S 6.2]{bobadilla2022} to have an absolute grading as in \cite{mclean2019}. We will denote this homology theory by $\mathrm{HF}_\bullet(\varphi,+)$.

\begin{remark} \label{remark}
    The shift in Theorem \ref{main-result} differs from the shift of the conjecture in \cite{budur2020cohomology}. As explained in \cite[Remark 7.6]{bobadilla2022}, this partly is due to a mistake in the formula \cite[Theorem 5.41(3)]{mclean2019}. The other reason is that we are using Floer \emph{homology} as in \cite{bobadilla2022} instead of the Floer cohomology as in \cite{mclean2019}. We compute the correct shift for our setting in Lemma \ref{shift-computation}.
\end{remark}

In view of Theorem \ref{main-result} the conjecture can be restated in terms of Floer homology, with the shift updated as explained in Remark \ref{remark}:

\begin{conjecture}[Arc-Floer conjecture]
    Let $f \in \C\{z_0,\dots,z_n\}$ be a convergent power series such that $f(\mathbf{0})=0$ that defines an isolated hypersuface singularity at the origin. For every $m \geq 1$, there is an isomorphism
    \[
    H_c^{\bullet + n(2m+1)}(\X_m) \cong \mathrm{HF}_\bullet(\varphi^m,+).
    \]
\end{conjecture}

The paper is structured as follows. In Section 2 we establish the notation and basic numerical and combinatorial invariants of plane curve singularities that will be used throughout the paper. In Section 3 we determine the connected components of the contact loci in terms of the resolution graph of $C$, and we study their topology. This essentially solves the embedded Nash problem \cite[\S 1.3]{budur2022nash}, \cite[Remark 2.8]{ein2004} for plane curves, generalizing \cite[Theorem 1.21]{budur2022nash}. In Section 4 we study the degeneration properties of the McLean spectral sequence \cite[Proposition 6.3]{bobadilla2022}, \cite[App. C]{mclean2019} of $\varphi^m$, which converges to $\mathrm{HF}_\bullet(\varphi^m,+)$. We deform the monodromy iterate $\varphi^m$ in a way that makes the McLean spectral sequence of the deformed map $\dphi$ degenerate at the first page. By invariance under isotopies, this computes the Floer homology of $\varphi^m$.
Finally, in Section 5 we prove Theorem \ref{main-result} by comparing the results of Sections 3 and 4.

\subsection*{Acknowledgements.} We are grateful to Javier Fernández de Bobadilla for proposing this problem to us. We are also grateful to Nero Budur, Javier Fernández de Bobadilla, Tomasz Pe{\l}ka and Pablo Portilla for useful discussions and answering our questions during the development of the paper. Finally, we wish to thank the referees, whose suggestions have been of great help to improve both the exposition and the results.

\noindent
J. de la Bodega was supported by PRE2019-087976 from the Ministry of Science of Spain and partially supported by G097819N, G0B3123N from FWO. E. de Lorenzo Poza was supported by 1187423N from FWO, Research Foundation, Flanders.



\section{Preliminaries on plane curves}
\label{curves}

Let $f = f^{[1]} \cdots f^{[b]} \in \C \llbracket x, y \rrbracket$ be the decomposition into irreducible factors of a reduced formal power series. Let $(C,0) = V(f)$ be the associated formal plane curve germ, which has an isolated singularity since $f$ is reduced. Similarly, let $(C^{[j]},0) = V(f^{[j]})$, $j = 1,\ldots,b$ be the irreducible components of $(C,0)$, usually known as the \emph{branches} of $C$.

\begin{notation}
    Throughout the paper we will attach some data to each branch $(C^{[j]},0)$. When we need to indicate the branch associated to the data, we will use a superscript $[j]$, $j = 1,\ldots,b$. Nevertheless, if the branch is clear from the context we will omit the superscript to make the notation cleaner. 
\end{notation}

\subsection{Puiseux data and resolution data} \label{subsec:Puiseux-and-res-data}
For each $j = 1,\ldots,b$, choose a coordinate system adapted to $(C^{[j]},0)$ as follows. Let $L_0$ be a smooth curve through the origin that has the same tangent direction as $(C^{[j]},0)$, and let $L_0'$ be a smooth curve through the origin that meets $L_0$ transversely. By the formal inverse function theorem, there is a coordinate system $(x_0, y_0)$ around the origin such that $L_0 = \{y_0 = 0\}$ and $L_0^{\prime} = \{x_0 = 0\}$. Without loss of generality, we may assume that if two branches $C^{[j]}, C^{[j']}$ have the same tangent direction, then their coordinate systems $(x_0^{[j]}, y_0^{[j]}), (x_0^{[j']}, y_0^{[j']})$ are equal.

In the coordinates $(x_0, y_0)$, the Newton-Puiseux theorem says that there exists a power series $\varphi(\tau) = \sum \varphi_k \tau^k \in \C \llbracket \tau \rrbracket$ such that
\begin{equation} \label{puiseux-decomposition}
f^{[j]}(x_0,y_0) = u(x_0,y_0) \prod_{\xi \in \bmu_{\mult(f^{[j]})}} \left(y_0 - \varphi(\xi (x_0)^{1/\mult(f^{[j]})})\right),
\end{equation}
where $u$ is a unit in $\C \llbracket x_0, y_0 \rrbracket$ and $\bmu_n$ is the group of $n$-th roots of unity in $\C$, see \cite[Theorem 5.1.7]{dejong2000}. In other words, we may parametrize $C^{[j]}$ as $(\tau^{\mult(f^{[j]})}, \varphi(\tau))$ in the coordinates $(x_0, y_0)$. We fix one such power series $\varphi^{[j]}$ for each $j = 1,\ldots,b$ for the rest of this paper.

\begin{definition} \label{characteristic-exponents}
    The \emph{characteristic exponents} of $(C^{[j]},0)$ are defined inductively as follows,
    \begin{align*}
        r_1 &= \mult(f^{[j]}), & k_1 &= \min \left\{k \mid \varphi_k \neq 0 \text{ and } r_1 \text{ does not divide } k\right\}, \\
        r_{i+1} &= \gcd(k_i, r_i), & k_{i+1} &= \min \left\{k \mid \varphi_k \neq 0 \text{ and } r_{i+1} \text{ does not divide } k \right\}.
    \end{align*}
    We denote by $g$ the smallest integer such that $k_{g + 1} = \infty$. The irreducibility of $f^{[j]}$ implies that $r_{g+1} = 1$. We also set
\[
\kappa_{i} = k_i - k_{i-1} \quad \text{ and } \quad (\hkappa_i, \hr_i) = (\kappa_i/r_{i+1}, r_i/r_{i+1}),
\]
where we put $k_0 = 0$ by convention. The pairs $(\hkappa_1, \hr_1), \ldots, (\hkappa_{g}, \hr_{g})$ are called the \emph{Newton pairs} of the branch $(C^{[j]},0)$, see \cite[p. 134]{Barroso2020}.

\end{definition}


Recall that a log resolution of the singularity $(C,0) \subset (\C^2,0)$ is a proper birational morphism $\mu: (Y, \mathbf{E}) \to (\C^2, 0)$ from a smooth variety $Y$ such that the total transform $\mu^\ast(C)$ is a simple normal crossing divisor. Such a resolution is obtained as a finite sequence of blow-ups. Let $\mathcal{E}$ be the set of irreducible components of the exceptional divisor $\mathbf{E} \coloneqq \mu^{-1}(0)$ and let $\mathcal{S} = \{\widetilde{C}^{[1]}, \ldots, \widetilde{C}^{[r]}\}$ be the set of irreducible components of $\widetilde{C} \coloneqq \mu^{-1}_*(C)$, the strict transform of $C$ under the resolution $\mu$. Associated to the resolution $\mu$ we define the following numerical invariants.

\begin{definition} \label{num-data}
    Let $\mu: (Y, \mathbf{E}) \to (\C^2, 0)$ be log resolution of $(C,0) = V(f)$ and let $E \in \mathcal{E} \cup \mathcal{S}$,
    \begin{enumerate}[label=(\roman{*})]
        \item The \emph{log discrepancy} at $E$ is $\nu_E \coloneqq \ord_{E}(K_{Y/\C^2}) + 1$.
        
        \item The \emph{multiplicity} of $f$ at $E$ is $N_E \coloneqq \ord_{E}(f)$. Similarly, for each $j = 1, \ldots, b$ the multiplicity of $f^{[j]}$ at $E$ is $N_E^{[j]} \coloneqq \ord_{E}(f^{[j]})$.

        \item Fix an integer $m \geq 1$. If $N_E$ divides $m$ we say that $E$ is an \emph{$m$-divisor} and denote $m_E \coloneqq m / N_E$.

        \item The log resolution $\mu$ is said to be \emph{$m$-separating} if $N_E + N_F > m$ for any $E, F \in \mathcal{E} \cup \mathcal{S}$ such that $E \cap F \neq \varnothing$.
    \end{enumerate}
\end{definition}

\subsection{Describing the resolution}

The algorithm to obtain the minimal log resolution of a plane curve singularity $(C,0)$ is classical, see for instance \cite[p. 522-531]{brieskorn1986}, \cite[\S 5.3, \S 5.4]{dejong2000} or \cite[\S 1.4]{Barroso2020}. The strategy is to blow up the points at which the total transform of the curve is not a simple normal crossing divisor until eventually there are no more such points. From the minimal log resolution we may obtain the \emph{minimal} $m$-separating log resolution by repeatedly blowing up the intersection of adjacent divisors which do not satisfy the $m$-separating condition, see \cite[Lemma 2.9]{budur2020cohomology}. From now on we fix $\mu: (Y, \mathbf{E}) \to (\C^2, 0)$ to be the minimal $m$-separating log resolution of $(C,0)$.

The total transform of $(C,0)$ is a simple normal crossing divisor whose structure we encode in the dual \emph{resolution graph}. This is an undirected graph whose set of vertices is $\mathcal{E} \cup \mathcal{S}$ (i.e. the exceptional divisors and the strict transforms of the branches of $C$), and two vertices are connected by an edge if and only if the corresponding divisors have nonempty intersection. The resolution graph is a tree, i.e. it contains no cycles.
Apart from Definition \ref{num-data}, we are going to associate to each $E \in \mathcal{E}$ some numerical data and coordinate systems, see Definition \ref{newton-pairs-E}. These will be used to give formulas for the log discrepancies (Proposition \ref{log-discrepancies}), the multiplicities (Proposition \ref{multiplicities}) and the leading coefficients of some power series (Lemma \ref{ac-equations}). Although the numerical data associated to $E$ may be recovered from any of the multiple combinatorial descriptions of the resolution, the coordinate systems that we will use are defined on partial resolutions that appear \emph{during} the resolution process. Therefore, to be able to introduce these coordinates we are going to give a more thorough recall of how the resolution algorithm works. First, let us introduce some notation for the kind of divisors that appear during the resolution process.

\begin{definition} \label{divisors-between}
    Let $L, L'$ be prime divisors in a smooth surface intersecting transversely at a point $P$. The \emph{divisors between $L$ and $L'$} are (the strict transforms of) the exceptional divisors that result from blowing up $P$ and, inductively, further intersection points of $L, L'$ and previous exceptional divisors.
\end{definition}

\begin{proposition} \label{blowup-formula}
    Let $L, L'$ be prime divisors in a smooth surface intersecting transversely at a point $P$. Set $E_{(1,0)} = L$, $E_{(0,1)} = L'$ and inductively define $E_{(\kappa+\kappa', r+r')}$ to be the exceptional divisor of the blow-up of the intersection of $E_{(\kappa,r)}$ and $E_{(\kappa', r')}$. This establishes a bijection between pairs of coprime numbers $(\kappa,r)$ with $\kappa, r \geq 1$, and divisors between $L$ and $L'$.

    Furthermore, if $x, y$ are local coordinates around $P$ for which $L = \{y=0\}$ and $L' = \{x=0\}$ then a local equation for the minimal composition of blow-ups that makes $E_{(\kappa,r)}$ appear is $(x, y) = (\widetilde{x}^r \widetilde{y}^a, \widetilde{x}^\kappa \widetilde{y}^b)$, where $a, b$ are the unique integers such that $a\kappa - br = (-1)^n$, $0 \leq a \leq r$, $0 \leq b \leq \kappa$, and $n$ is the number of divisions in the Euclidean algorithm to compute the greatest common divisor of $\kappa$ and $r$. In these coordinates $E_{(\kappa,r)} = \{\widetilde{x} = 0\}$.
\end{proposition}

\begin{proof}
    The fact that all pairs of coprime numbers can be obtained inductively by combining previous pairs $(\kappa,r)$ and $(\kappa',r')$ into $(\kappa+\kappa', r+r')$ is a classical construction known as the Stern-Brocot tree, see \cite{gibbons2006}. It can also be found in the literature under the name of Farey sequence or Farey sums. This establishes the first part of the theorem, since this is by definition how divisors between $L$ and $L'$ are constructed.

    
    Consider $(0,1)$ and $(1,0)$ to be the left and right generating nodes of the Stern-Brocot tree respectively, see \cite[Figure 1]{gibbons2006}. To a pair $(\kappa,r)$ we may associate the sequence $q_1, \ldots, q_n$ of quotients in the Euclidean algorithm of $\kappa$ and $r$ (or equivalently, the numbers defining the continued fraction expression for $\tfrac{\kappa}{r}$):
    \[
    c_{j-1} = q_j c_j + c_{j+1}, \text{ with } c_0 = \kappa, c_1 = r \text{ and } 0 \leq c_{j+1} < c_j, \text{ until } c_{n-1} = q_n c_n.
    \]
    We convene that $c_{n+1} = 0$. Recall that the Euclidean algorithm may be written in matrix form as follows:
    \[
    \begin{pmatrix} c_{j-1} \\ c_j \end{pmatrix} =  \begin{pmatrix}
        q_j & 1 \\ 1 & 0
    \end{pmatrix} \begin{pmatrix} c_j \\ c_{j+1} \end{pmatrix}, \quad \text{for all } j = 1, \ldots, n.
    \]

    The sequence $q_1, \ldots, q_n$ indicates the path to the pair $(\kappa, r)$ in the Stern-Brocot tree --- first step $q_1$ times to the right, then $q_2$ times to the left, and so on. Under the bijection between coprime pairs and divisors, stepping right in the Stern-Brocot tree corresponds to taking local coordinates in which the blow-up is given by $(x,y) = (x', x' y')$, and stepping left corresponds to taking local coordinates in which the blow-up is given by $(x,y) = (x' y', y')$. Let us define the following matrix notation for the exponents of a composition of blow-ups in this type of coordinates:
    \[
    (x, y) = (\widetilde{x}^\alpha \widetilde{y}^\beta, \widetilde{x}^\gamma \widetilde{y}^\delta) \quad \text{corresponds to the matrix} \quad \begin{pmatrix} \alpha & \beta \\ \gamma & \delta \end{pmatrix}.
    \]

    The composition of $q_1$ blow-ups of the form $(x,y) = (x', x' y')$, followed by $q_2$ blow-ups of the form $(x,y) = (x' y', y')$, and so on corresponds to the following product of matrices:
    \begin{equation*}
    M = \begin{pmatrix}
        1 & 0 \\ q_1 & 1
    \end{pmatrix}
    \begin{pmatrix}
        1 & q_2 \\ 0 & 1
    \end{pmatrix}
    \begin{pmatrix}
        1 & 0 \\ q_3 & 1
    \end{pmatrix}
    \cdots
    \begin{pmatrix}
        1 & 0 \\ q_n & 1
    \end{pmatrix},
    \end{equation*}
    where the last matrix has to be transposed if $n$ is even. Note that if $n$ is odd, the equation of the last exceptional divisor is $\widetilde{x} = 0$, whereas if $n$ is even the equation is $\widetilde{y} = 0$. To avoid this ambiguity, we convene that if $n$ is even then will swap the variables after the last blow-up, so that the equation of last exceptional divisor is always $\widetilde{x} = 0$. This alters the matrix of exponents by swapping the columns, i.e. by a multiplication on the right by the $2\times2$ matrix which has $1$'s on the anti-diagonal and $0$'s on the diagonal. Taking this into account, we see that the matrix of exponents of the minimal composition of blow-ups that makes $E_{(\kappa,r)}$ appear, with the convention of swapping the last variables to ensure that the equation of last exceptional divisor is always $\widetilde{x} = 0$, is
    \begin{equation} \label{kappa-r-b-a-matrix}
    M \begin{pmatrix}
        0 & 1 \\ 1 & 0
    \end{pmatrix}^{n+1} = 
    \begin{pmatrix}
        0 & 1 \\ 1 & 0
    \end{pmatrix}
    \begin{pmatrix}
        q_1 & 1 \\ 1 & 0
    \end{pmatrix}
    \begin{pmatrix}
        q_2 & 1 \\ 1 & 0
    \end{pmatrix}
    \cdots
    \begin{pmatrix}
        q_n & 1 \\ 1 & 0
    \end{pmatrix}
    =
    \begin{pmatrix}
        r & a \\ \kappa & b
    \end{pmatrix},
    \end{equation}
    where the second equality comes from the Euclidean algorithm in matrix form, and $a, b$ are simply integers coming from the matrix product. Taking determinants shows $a\kappa - br = (-1)^n$, which determines $a$ and $b$ uniquely by the B{\'e}zout identity.
\end{proof}

\noindent
With this notation in hand, we are ready to describe the steps of the resolution algorithm.

\textit{$i$-th iteration of the resolution algorithm:}

\begin{enumerate}[label=(\arabic{*}), start=0, wide]
    \item \label{resol-starting-situation} \textbf{Starting situation.}
    Recall that at the beginning we fixed a coordinate system $(x_0, y_0)$ for each branch $C^{[j]}$, whose axes we called $L_0, L_0'$. We also assumed that these coordinates are equal for branches with the same tangent direction. After the $(i-1)$-st step of the resolution, we have the following data:
    \begin{itemize}
        \item A proper birational morphism $\mu_{i-1}: Y_{i-1} \to \C^2$ such that the strict transform of each branch $C^{[j]}$ intersects exactly one exceptional divisor, which we call $R_{i-1} = R_{i-1}^{[j]}$, at a point $Q_{i-1}=Q_{i-1}^{[j]}$.
    
        \item For each branch $C^{[j]}$, a coordinate system $(x_{i-1}, y_{i-1})$ around $Q_{i-1}$ such that $R_{i-1} = L_{i-1}' \coloneqq \{x_{i-1} = 0\}$. If the strict transform $\mu_{i-1,*}^{-1}(C^{[j]})$ is \emph{not} tangent to $R_{i-1}$, then the coordinates are chosen so that $\mu_{i-1,*}^{-1}(C^{[j]})$ is tangent to $L_{i-1} \coloneqq \{y_{i-1} = 0\}$.
    \end{itemize}
    
    \begin{figure}[ht] 
        \centering
        \makebox[\textwidth][c]{\resizebox{0.6\linewidth}{!}{%
    \fontsize{20pt}{20pt}\selectfont
    \def\svgheight{0.5cm}
    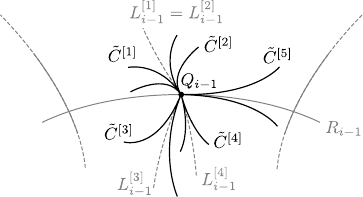
}}
        \caption{Sketch of the situation before the $i$-th iteration of the resolution algorithm.}
        \label{fig:res-algorithm-step-before}
    \end{figure}

    \item \textbf{Blow-ups.}
    For each $j = 1,\ldots,b$, if the strict transform $\mu_{i-1,*}^{-1}(C^{[j]})$ is smooth and does not intersect the strict transform of any other branch, then the resolution algorithm is done with $C^{[j]}$. Otherwise, we repeatedly blow-up the intersection of $L_{i-1}$ and $L_{i-1}'$ and the new divisors that appear (i.e. we do blow-ups with exceptional divisors between $L_{i-1}$ and $L_{i-1}'$ in the sense of Definition \ref{divisors-between}). The divisors that appear in this way get labeled by a pair of coprime numbers $(\kappa, r)$ as in Proposition \ref{blowup-formula}.
    
    After a finite number of blow-ups, whose composition we denote by $\widetilde{\mu}_i: Y_i \to Y_{i-1}$ and $\mu_i \coloneqq \mu_{i-1} \circ \widetilde{\mu}_i$, each strict transform $\mu_{i,*}^{-1}(C^{[j]})$ intersects only one exceptional divisor, which we call $R_i=R_i^{[j]}$. In fact, $R_i$ is precisely the divisor $E_{(\hkappa_i, \hr_i)}$ between $L_{i-1}$ and $L_{i-1}'$, see \cite[p. 522-524]{brieskorn1986}, \cite[Th. 5.3.14]{dejong2000}.
    
    Note that if $Q_{i-1}^{[j]} = Q_{i-1}^{[j']}$, then some of the divisors $E_{(\kappa, r)}^{[j]}$ are common to both branches $C^{[j]}$ and $C^{[j']}$ (see \cite[p. 528]{brieskorn1986} for a different interpretation of the same phenomenon). Indeed, we have $E_{(0,1)}^{[j]} = E_{(0,1)}^{[j']} = R_{i-1}$ and $E_{(1,1)}^{[j]} = E_{(1,1)}^{[j']}$ is the exceptional divisor of the blow-up at $Q_{i-1}$. Therefore, for all coprime pairs $(\kappa, r)$ with $\kappa/r \leq 1$ we also have $E_{(\kappa,r)}^{[j]} = E_{(\kappa,r)}^{[j']}$. On the other hand, $E_{(1,0)}^{[j]}$ is the strict transform of $L_{i-1}^{[j]}$, which is different for each tangent direction at $Q_{i-1}$. 
    
    To summarize, the exceptional divisors of $\widetilde{\mu}_i$ form a chain from $E_{(0,1)}$ to $E_{(1,1)}$, which then splits into several chains of divisors, one for each tangent direction of $\mu_{i-1,*}^{-1}(C)$ at $Q_{i-1}$, see Figure \ref{fig:res-algorithm-step-before} and Figure \ref{fig:res-algorithm-step-after}.
    
    \begin{figure}[ht] 
        \centering
        \makebox[\textwidth][c]{\resizebox{0.8\linewidth}{!}{%
    \fontsize{20pt}{20pt}\selectfont
    \def\svgheight{0.5cm}
    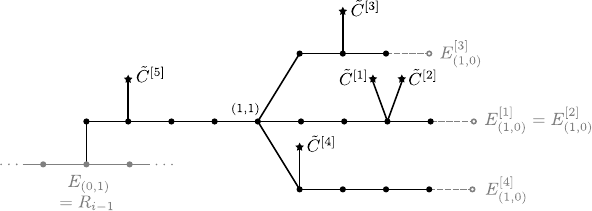
}}
        \caption{Sketch of the part dual graph of $\mu_i$ that has been added during the $i$-th iteration, i.e. the dual graph of $\widetilde{\mu}_i$. The dashed lines have been added to show where the coordinate axes end up.}
        \label{fig:res-algorithm-step-after}
    \end{figure}

    \item \textbf{New coordinates.}
    To get back to the starting situation we just need to give the coordinate system $(x_i, y_i)$. Let $L_i' = R_i$. If $\mu_{i,*}^{-1}(C^{[j]})$ is \emph{not} tangent to $R_i$ let $L_i$ be a smooth curve tangent to $\mu_{i,*}^{-1}(C^{[j]})$, and otherwise let $L_i$ be any smooth curve transverse to $L_{i}'$.
    
    In the local coordinates $(\widetilde{x}, \widetilde{y})$ given by Proposition \ref{blowup-formula}, we have $R_i = L_i' = \{\widetilde{x} = 0\}$. We may assume that $L_i$ is the line $\widetilde{y} = A_i+B_i\widetilde{x}$ for some $A_i,B_i \in \C$. Thus, after the change of coordinates $(x_i,y_i)=(\widetilde{x},\widetilde{y}-A_i-B_i\widetilde{x})$, the coordinate axes for $(x_i, y_i)$ are $L_i$ and $L_i'$.
    
    We remark that $A_i \neq 0$ because by assumption $\mu_{i,*}^{-1}(C^{[j]})$ intersects a unique exceptional divisor. Furthermore, if $\mu_{i,*}^{-1}(C^{[j]})$ is tangent to $R_i$, then we have some freedom in the choice of the direction of $L_i$, which is reflected in $B_i$ (cf. \cite[Definition 1.3.31]{Barroso2020} for a similar situation). Either way, the coordinates $(x_i, y_i)$ are related to the previous ones by the formula
    \begin{equation} \label{next-rupture-component}
        \begin{pmatrix}
        x_{i-1} \\[3pt] y_{i-1}
        \end{pmatrix} = 
        \begin{pmatrix}x_i^{\hr_i} \left(A_i + B_i x_i + y_i\right)^{a_i} \\[3pt] x_i^{\hkappa_i} \left(A_i + B_i x_i + y_i\right)^{b_i}
        \end{pmatrix},
    \end{equation}
    by Proposition \ref{blowup-formula}, where $a_i, b_i$ are the unique integers with 
    \[
    a_i \hat{\kappa}_i - b_i \hat{r}_i = (-1)^{n_i}, \quad 0 \leq a_i < \hat{r}_i, \quad 0 \leq b_i < \hat{\kappa}_i,
    \]
    and $n_i$ is the number of divisions in the Euclidean algorithm of $(\hkappa_i, \hr_i)$. Note that the coordinates $(x_i, y_i)$ depend only on the tangent direction of $\mu_{i-1,*}^{-1}(C^{[j]})$ at $Q_{i-1}$.
\end{enumerate}

After finitely many iterations of steps (0), (1) and (2), this process finishes, and we obtain the minimal embedded resolution of $(C,0)$. The extra divisors that appear in the minimal $m$-separating log resolution come from further blowing up the intersection of exceptional divisors and strict transforms in the minimal resolution, so they can also be described in terms of the $(\kappa,r)$ labels associated to the coordinates $(x_i, y_i)$.

Now that we have described the resolution, let us define some concepts that will be useful in the following sections.

\begin{definition} \label{newton-pairs-E}
    Let $E \in \mathcal{E}$ be a divisor in the minimal $m$-separating log resolution $\mu: Y \to \C^2$ (or more generally any divisor that appears after a finite sequence of blow-ups $\mu: Y \to \C^2$). Choose a smooth curve germ $\widetilde{D} \subset Y$ that intersects $E$ transversely and meets no other divisor. The direct image $D \coloneqq \mu_*\widetilde{D}$ is an irreducible germ of a plane curve, and hence it has some Newton pairs associated to it (see the beginning of this section). These pairs do not depend on the choice of $\widetilde{D}$; we call them the \emph{Newton pairs} associated to $E$ and denote them by $(\hkappa_i^E, \hr_i^E)$, $i = 1,\ldots,g^E$. We also define $r_i^E, \kappa_i^E, k_i^E$ and $R_i^E$ for $i \leq g^E$, and $Q_i^E$ for $i \leq g^E - 1$, to be equal to the corresponding data of $D$, see Definition \ref{characteristic-exponents} and Step \ref{resol-starting-situation} of the resolution algorithm. Note that $R_{g^E}^E = E$. The point $Q_i^E$ is not defined for $i = g^E$ because it would depend on the choice of $D$. We fix the notation $(\hkappa^E, \hr^E) \coloneqq (\hkappa^E_{g^E}, \hr^E_{g^E})$, since we will use those numbers frequently.
    
    Finally, the resolution of $D$ gives rise to coordinates around each $Q_i^E$ which we denote by $(x_i^E, y_i^E)$, $i = 1, \ldots, g^E - 1$, and they are related with each other by equation \eqref{next-rupture-component}. We also introduce coordinates $(x_{g^E}^E, y_{g^E}^E)$ using directly the change of coordinates given by Proposition \ref{blowup-formula},
    \begin{equation} \label{last-blow-up-E}
    (x_{g^E-1}^E, y_{g^E-1}^E) = \left((x_{g^E}^E)^{\hr^E}(y_{g^E}^E)^{a_{g^E}^E},\ (x_{g^E}^E)^{\hkappa^E}(y_{g^E}^E)^{b_{g^E}^E}\right).
    \end{equation}
    We do not shift these coordinates because, unlike in the previous cases, the point $Q_i^E$ is not defined for $i = g^E$.
\end{definition}

\begin{remark}
    Note that the Newton pairs associated to $E \in \mathcal{E}$ are \emph{not} enough to distinguish $E$ from the other divisors --- in order to do that we need the extra information of the points that are being blown up and the coordinate axes that are being used to define the $(\kappa,r)$ coordinates. 
\end{remark}

We finish the subsection with a few definitions that will be useful when we want to refer to the different parts of the resolution graph, see Figure \ref{fig:example-resolution} for an example.

\begin{definition}
\label{parts-resolution-graph}
\begin{enumerate}[label=(\roman{*}), wide] 
    \item We say $R \in \mathcal{E}$ is a \emph{rupture divisor} if it intersects at least three other divisors in $\mathcal{E} \cup \mathcal{S}$, i.e. if the corresponding vertex in the dual graph of $\mu^*(C)$ has valence greater or equal than $3$. We denote by $\mathcal{R} \subset \mathcal{E}$ the set of rupture divisors. Among these, we denote by $\mathcal{R}_1 \subset \mathcal{R}$ the set of rupture divisors $R$ such that $(\hkappa^R, \hr^R) = (1,1)$, and $\mathcal{R}_0 \coloneqq \mathcal{R} \setminus \mathcal{R}_1$. Note that $\mathcal{R}_0$ is the set of divisors which are also rupture divisors in the minimal log resolution of some branch $C^{[j]}$.

    \item Let $R \in \mathcal{R}$ and consider a simple path starting at $R$ in the dual graph of $\mu^*(C)$, i.e. a sequence of pairwise distinct divisors $R = E_0,E_1,\dots,E_n \in \mathcal{E} \cup \mathcal{S}$ such that $E_{i} \cap E_{i+1} \neq \varnothing$ for every $i=0,\ldots,n-1$. Suppose that $\mathrm{val}(E_i) = 2$ for $i = 1, \ldots, n-1$ and $\mathrm{val}(E_n) \neq 2$, where $\mathrm{val}(E)$ denotes the valence of the vertex $E$. There are two possible situations.
    \begin{enumerate}[label=(\arabic{*}), leftmargin=2cm]
        \item If $\mathrm{val}(E_n) = 1$ and $E_n$ is not a strict transform, we say that the path is an \emph{end} of the dual graph, the vertex $E_n$ is called a \emph{leaf}, and we denote by $\mathcal{F}_R \subset \mathcal{E}$ the set of divisors in the path between the rupture divisor $R$ and the leaf, \emph{including} both of them. If there are no ends attached to $R$ we simply set $\mathcal{F}_R = \{R\}$. See Lemma \ref{end-well-defined} for a proof that there is (almost always) at most one end for each $R$, so that $\mathcal{F}_R$ is well-defined, and a necessary small adjustment to the definition in the edge case where there are two ends instead of one. 


        \item If $E_n$ is a rupture divisor or a strict transform, i.e. $E_n = S \in \mathcal{R} \cup \mathcal{S}$, then we say that the path is a \emph{trunk} of the dual graph, and denote by $\mathcal{T}_{R,S} \subset \mathcal{E}$ the set of divisors in the path, \emph{not} including either $R$ or $S$. If $R, S \in \mathcal{R} \cup \mathcal{S}$ are not connected by a path of divisors of valence $2$, we set $\mathcal{T}_{R,S} = \varnothing$. We denote by $\mathcal{N}_R$ the set of $S \in \mathcal{R} \cup \mathcal{S}$ such that $R$ is connected to $S$ by a path of divisors of valence $2$, and call the $S \in \mathcal{N}_R$ the \emph{neighbors} of $R$.
    \end{enumerate}

    \item \label{special-case} If some divisors in an end $\mathcal{F}_R$ are $m$-divisors, there is one of them which is closest to $R$ (which will be $R$ itself if $R$ is an $m$-divisor). We denote this divisor by $E_{\dom}(m,R)$. The notation ``dom'' comes from the fact that arcs lifting to this divisor dominate all arcs lifting to $\mathcal{F}_R$, see Theorem \ref{arc-contact-loci-decomposition}. This is also well defined in the special case explained in Lemma \ref{end-well-defined}, i.e. when all branches of $C$ share the same divisor $R_1^{[j]}$, since in this case either $R_1^{[j]}$ is an $m$-divisor or there are $m$-divisors in at most one of the two ends, see \cite[Proposition 7.16]{budur2022nash} or Remark \ref{special-case-multiplicities}. 

    We denote by $\mathcal{F}_{m,R}$ set of divisors in the path from $E_\dom(m,R)$ to the leaf of $\mathcal{F}_R$. If $R$ is an $m$-divisor then $\mathcal{F}_{m,R} = \mathcal{F}_R$ (we use this as a definition in the exceptional case of a shared $R_1^{[j]}$). If there are no $m$-divisors in $\mathcal{F}_R$, we set $\mathcal{F}_{m,R} = \varnothing$.
\end{enumerate}
\end{definition}

\begin{lemma} \label{end-well-defined}
    Let $R \in \mathcal{R}$ be a rupture divisor in the minimal $m$-separating resolution.
    \begin{enumerate}[label=(\roman{*})]
        \item If $R \in \mathcal{R}_1$, then there are no ends coming out of $R$, i.e. $\mathcal{F}_R = \{R\}$.
        \item \label{special-case-ends} If $R$ is the first rupture divisor for all branches $C^{[j]}$, i.e. $R = R_1^{[j]}$ for all $j = 1, \ldots, b$, then there are two ends coming out of $R$. We denote them by $\mathcal{F}_{R,1}$ and $\mathcal{F}_{R,2}$ (the order is irrelevant for us).
        \item In any other case, there is at most one end coming out of $R$, so the notation $\mathcal{F}_R$ is well defined.
    \end{enumerate} 
\end{lemma}

\begin{proof}
    If $R \in \mathcal{R}_1$, there are two possible situations. If $R$ is the root of the resolution tree (i.e. the exceptional divisor of the very first blow-up), then every path starting from $R$ corresponds to a tangent direction of $C$ at the origin. Therefore for every path starting from $R$ there is at least one branch $C^{[j]}$ whose first rupture component $R_1^{[j]}$ lies in that path, and hence the path is not an end. 
    
    If $R \in \mathcal{R}_1$ is not the root of the tree, then let $Q$ be the point whose blow-up is $R$. For every tangent direction of (the strict transform of) $C$ at $Q$ there is a path starting from $R$, and the same argument as before applies. There is also one other path starting at $R$, namely the one joining $R = E_{(1,1)}$ with the divisor $E_{(0,1)}$ (this is the divisor in which $Q$ lies). Since $E_{(0,1)}$ is a rupture divisor (by definition), this path is not an end either. This proves (i).

    Now let $R \in \mathcal{R}_0$. Since we are working with the minimal $m$-separating resolution there have been no unnecessary blow-ups in the resolution process, so there is a branch $C^{[j]}$ and an integer $i = 1, \ldots, g^{[j]}$ such that $R = R_i^{[j]}$ (these $C^{[j]}$ and $i$ are in general not unique). In particular, $R$ is a divisor between $L_{i-1}^{[j]}$ and $L_{i-1}^{\prime[j]}$, so we have two distinguished paths coming out of $R$: the ones going to $L_{i-1}^{[j]}$ and $L_{i-1}^{\prime[j]}$. 
    
    Any path starting from $R$ other than these two comes from the resolution of some branch, so the path eventually reaches another rupture component or a strict transform, and hence is not an end. The two distinguished paths may or may not be ends depending on whether there are other rupture components along them. In fact, $L_{i-1}^{\prime[j]}$ is always a rupture component if $i > 1$, so the only case in which both paths are ends is when $i = 1$ and there are no other rupture divisors between $L_0$ and $L_0'$, i.e. when $R$ is the first rupture divisor for all branches of $C$. This proves (ii) and (iii).
\end{proof}




\begin{definition}
    Let $E$ be a divisor in $\mathcal{E} \cup \mathcal{S}$. We define $c(E)$ to be the greatest common divisor of all the multiplicities of $f$ at the divisors that meet $E$ (including $E$ itself), that is
    \[
    c(E) \coloneqq \gcd(N_F \mid F \in \mathcal{E} \cup \mathcal{S}, \ E \cap F \neq \varnothing).
    \]
\end{definition}

\begin{proposition} \label{cE-independence}
    \begin{enumerate}[(i)]
        \item \label{cE-simply} Let $E \in \mathcal{E} \cup \mathcal{S}$. Then $N_E$ divides $\sum_{\substack{F \in (\mathcal{E} \cup \mathcal{S}) \setminus \{E\}, E \cap F \neq \varnothing}} N_F$, the sum of the multiplicities of all divisors that intersect $E$. In particular, removing one divisor $E' \neq E$ from the computation of $c(E)$ does not change the result
        , i.e.
        \[
        c(E) = \gcd(N_F \mid F \in (\mathcal{E} \cup \mathcal{S}) \setminus \{E'\}, \ E \cap F \neq \varnothing).
        \]
        \item \label{cE-independence-trunk} If $R,S \in \mathcal{R} \cup \mathcal{S}$ are neighbors, and the divisors $E \in \mathcal{T}_{R,S}$ and $F \in \mathcal{T}_{R,S} \cup \{R,S\}$ have nonempty intersection, then $c(E)=\gcd(N_E,N_F)$. Hence for any two $E,F \in \mathcal{T}_{R,S}$ we have $c(E) = c(F)$ (recall that the trunk $\mathcal{T}_{R,S}$ does \emph{not} include the rupture divisors $R, S$).
        
        \item If $R \in \mathcal{R}$ and $E \in \mathcal{F}_R \setminus \{R\}$, then $c(E)=N_L$, where $L$ is the leaf of $\mathcal{F}_R$.
    \end{enumerate}
\end{proposition}
\begin{proof}
    Note that (ii) and (iii) are immediate from (i). To prove (i), we just need to observe that $0 = \mathrm{div}(f \circ \mu) \cdot E = N_E E^2 + \sum_{\substack{F \in (\mathcal{E} \cup \mathcal{S}) \setminus \{E\}, E \cap F \neq \varnothing}} N_F$. 
\end{proof}



This motivates the following definition:

\begin{definition} \label{gcds}
    \begin{enumerate}[(i)]
        \item For $R,S \in \mathcal{R} \cup \mathcal{S}$ neighbors, we define $c(\mathcal{T}_{R,S})$ to be equal to $c(E)$ for any divisor $E \in \mathcal{T}_{R,S}$.
        \item For $R \in \mathcal{R}$, we define $c(\mathcal{F}_R)$ to be equal to $c(E)$ for any divisor $E \in \mathcal{F}_R \setminus \{R\}$.
    \end{enumerate}
\end{definition}

\begin{definition} \label{left-right-up}
    Given an exceptional divisor $E \in \mathcal{E}$, we split the strict transforms of the branches of $C$ into several groups as follows. Let $\Gamma$ be the dual graph of $\mu^\ast(C)$, and denote by $\Gamma \setminus E$ the graph obtained by deleting the vertex corresponding to $E$.

    \begin{enumerate}[(i)]
        \item If $E$ is not the root of $\Gamma$ (i.e. the vertex corresponding to the exceptional divisor of the first blow-up), then we call the connected component of $\Gamma \setminus E$ that contains the root the \emph{left} component of $\Gamma \setminus E$. We say that the $\widetilde C^{[j]}$ that lie on the left component of $\Gamma \setminus E$ are \emph{to the left} of $E$. We denote $\mathcal{S}^E_\leftarrow \coloneqq \{j \in \{1,\ldots,b\} \mid \widetilde C^{[j]} \text{ lies to the left of } E\}$.

        \item Any component of $\Gamma \setminus E$ containing divisors $F$ with $R^F_{g^E - 1} = R^E_{g^E - 1}$ and $\kappa^F_{g^E}/r^F_{g^E} > \kappa^E_{g^E}/r^E_{g^E}$ is said to be a \emph{right} component. Note that if $E \not\in \mathcal{R}_1$, then $\Gamma \setminus E$ has exactly one right component. The $\widetilde C^{[j]}$ that lie on a right component of $\Gamma \setminus E$ are said to be \emph{to the right} of $E$. We denote $\mathcal{S}^E_\rightarrow \coloneqq \{j \in \{1,\ldots,b\} \mid \widetilde C^{[j]} \text{ lies to the right of } E\}$.

        \item Let $U^E$ be the set of connected components of $\Gamma \setminus E$ different from the left and right components that we just defined. Clearly $U^E \neq \varnothing$ if and only if $E \in \mathcal{R}_0$. We denote $\mathcal{S}^E_u \coloneqq \{j \in \{1,\ldots,b\} \mid \widetilde C^{[j]} \text{ lies in the component } u \in U^E\}$. If $j \in \mathcal{S}^E_u$ for some $u \in U^E$ we say that $\widetilde C^{[j]}$ lies \emph{above} $E$.
    \end{enumerate}
    Note that these three cases are mutually exclusive and exhaustive, i.e. for a fixed exceptional divisor $E$ every strict transform $\widetilde C^{[j]}$ lies in exactly one connected component of $\Gamma \setminus E$.
\end{definition}

\begin{example}
    To illustrate the concepts that we have defined, consider the plane curve branches given by the following parametrizations:
    \begin{align*}
        C^{[1]} &\equiv (\tau^{42}, 2\tau^{105} + \tau^{133} + \tau^{136}), & 
        C^{[2]} &\equiv (\tau^6, \tau^{15} + \tau^{17}), \\
        C^{[3]} &\equiv (\tau^4, \tau^{10} + \tau^{17}), &
        C^{[4]} &\equiv (\tau^6(1 + \tau^3 + \tau^5), \tau^{15}(1 + \tau^3 + \tau^5)^2),
    \end{align*}
    and let $C = C^{[1]} \cup C^{[2]} \cup C^{[3]} \cup C^{[4]}$. The parametrizations have been chosen so that the resolution graph of $C$ looks like Figure \ref{fig:example-resolution}.

    Using a computer algebra software, it is possible to find a change of coordinates so that the branch $C^{[4]}$ also has a parametrization of the form $(\widetilde{\tau}^{r_1}, \varphi(\widetilde{\tau}))$. We did this with the following SageMath code: 
    \begin{verbatim}
        sage: R.<t> = PowerSeriesRing(QQ)
        sage: gx = t^6 * (1 + t^3 + t^5) + O(t^15)
        sage: r = gx.nth_root(6)
        sage: r_rev = r.reverse()
        sage: gx(t = r_rev)
        t^6 + O(t^15)
        sage: gy = t^15 * (1 + t^3 + t^5)^2
        sage: gy(t = r_rev)
        t^15 - 1/2*t^18 - 1/2*t^20 + 5/8*t^21 + 17/12*t^23 + O(t^24)
    \end{verbatim}
    The resulting parametrization is $C^{[4]} \equiv (\widetilde{\tau}^6, \widetilde{\tau}^{15} - \frac{1}{2} \widetilde{\tau}^{18} - \frac{1}{2} \widetilde{\tau}^{20} + \cdots)$. From here we get the characteristic exponents and Newton pairs of all branches.

    \begin{center}
    \begin{tabular}{c|c}
        {$\begin{aligned}
            k_1^{[1]} &= 105,\ & k_2^{[1]} &= 133,\ & k_3^{[1]} &= 136 \\
            \kappa_1^{[1]} &= 105,\ & \kappa_2^{[1]} &= 28,\ & \kappa_3^{[1]} &= 3 \\
            r_1^{[1]} &= 42,\ & r_2^{[1]} &= 21,\ & r_3^{[1]} &= 7 \\
            (5&,2) & (4&,3) & (3&,7)
        \end{aligned}$} \quad & \quad {$\begin{aligned}
            k_1^{[2]} &= 15,\ & k_2^{[2]} &= 17 \\
            \kappa_1^{[2]} &= 15,\ & \kappa_2^{[2]} &= 2 \\
            r_1^{[2]} &= 6,\ & r_2^{[2]} &= 3 \\
            (5&,2) & (2&,3)
        \end{aligned}$} \\[-1em] \\
        \hline
        \\[-0.3cm]
        {$\begin{aligned}
            k_1^{[3]} &= 10,\ & k_2^{[3]} &= 17 \\
            \kappa_1^{[3]} &= 10,\ & \kappa_2^{[3]} &= 7 \\
            r_1^{[3]} &= 4,\ & r_2^{[3]} &= 2 \\
            (5&,2) & (7&,2)
        \end{aligned}$} \quad & \quad {$\begin{aligned}
            k_1^{[4]} &= 15,\ & k_2^{[4]} &= 20 \\
            \kappa_1^{[4]} &= 15,\ & \kappa_2^{[4]} &= 5 \\
            r_1^{[4]} &= 6,\ & r_2^{[4]} &= 3 \\
            (5&,2) & (5&,3)
        \end{aligned}$}
    \end{tabular}
    \end{center}

    In order to draw the resolution graph we also need to know how each strict transform intersects the rupture components during the resolution process (this is the data of the $A_i, B_i \in \C$ that appear in \eqref{next-rupture-component}). We chose the parametrizations so that $A_1^{[1]} = 4$, but $A_1^{[2]} = A_1^{[3]} = A_1^{[4]} = 1$ (hence $C^{[1]}$ meets $R_1$ at a different point than the other branches) and $B_1^{[2]} = B_1^{[3]} = 0$, but $B_1^{[4]} = 1$ (hence $C^{[4]}$ meets $R_1$ at the same point as $C^{[2]}, C^{[3]}$ but with a different tangent direction). 

    \begin{figure}[ht] 
        \centering
        \makebox[\textwidth][c]{\resizebox{0.8\linewidth}{!}{%
    \fontsize{20pt}{20pt}\selectfont
    \def\svgheight{0.5cm}
    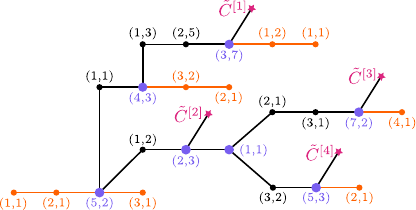
}}
        \caption{Minimal log resolution of the curve $C$. The rupture components are \textcolor{myblue}{blue}, the strict transforms are \textcolor{mypink}{pink}, the ends are \textcolor{myorange}{orange} (and they include the corresponding rupture divisor) and the trunks are black.}
        \label{fig:example-resolution}
    \end{figure}

    With this knowledge, we draw Figure \ref{fig:example-resolution}. Each exceptional divisor has been labeled with its last Newton pair, i.e. $E$ has been labeled $(\hkappa^E, \hr^E)$. The drawing has colors to separate the different parts of the graph according to Definition \ref{parts-resolution-graph}. Note how each rupture divisor has at most one end coming out of it, except the divisor labeled $(5,2)$ on the bottom left, which is the divisor $R_1^{[j]}$ for all branches $C^{[j]}$, so this is an example of the special case described in Lemma \ref{end-well-defined}.\ref{special-case-ends}.

    We can also see the components which are left, right and above each exceptional divisor, according to Definition \ref{left-right-up}. For example, consider the divisor $E$ to which $\widetilde{C}^{[2]}$ is connected (labeled $(2,3)$). Then the branch $\widetilde C^{[1]}$ is to the left of $E$, the branch $C^{[2]}$ is above $E$, and branches $\widetilde C^{[3]}$ and $\widetilde C^{[4]}$ are to the right of $E$. On the other hand, let $R$ be the unique rupture component labeled $(1,1)$, then $\widetilde C^{[1]}$ and $\widetilde C^{[2]}$ are to the left of $R$, whereas $\widetilde C^{[3]}$ and $\widetilde C^{[4]}$ are to the right of $R$.
\end{example}

\subsection{Computing the resolution data}

A key ingredient in our computations is the combination of the above explicit description of an embedded resolution with the Puiseux decomposition. The following lemma, which is part of the proof of \cite[Proposition 7.10]{budur2022nash}, gives a formula for the lift of the Puiseux factors $y_0 - \varphi(\xi (x_0)^{1/\mult(f^{[j]})})$ to different steps of the resolution.

\begin{notation} \label{arc-notation}
    Let $h \in \C \llbracket x,y \rrbracket$ be a power series and let $\gamma(t) = (\alpha(t), \beta(t))$ be an arc or parametrized curve in the plane. We will denote by $h(\gamma(t))$, or even $h(t)$ when the arc is understood from context, the result of substituting the series $\alpha(t), \beta(t)$ in $h$. The parameter $t$ will be used for arcs, and the parameter $\tau$ will be reserved for the curve $C$.
\end{notation}

\begin{lemma} \label{pxi-decomposition}
    Fix an irreducible component $(C^{[j]},0)$ of $(C,0)$. Let $\varphi \in \C \llbracket t \rrbracket$ be the series in the Puiseux decomposition \eqref{puiseux-decomposition} and for each $\xi \in \bmu_{r_1}$ let $P_{\xi, 0} \coloneqq y_0 - \varphi(\xi x_0^{1/r_1})$, which is a power series in $x_0^{1/r_1}$ and $y_0$. Use \eqref{next-rupture-component} to see $x_i, y_i$ as polynomials in $x_{i'}, y_{i'}$ for all $0 \leq i \leq i' \leq g$. Then for each $i = 1, \ldots, g$ and each $\xi \in \bmu_{r_i}$ we may express 
    \[
    P_{\xi,0} = y_0 \cdots y_{i-1} P_{\xi,i},
    \]
    where $P_{\xi,i}$ is a power series in $x_i^{1/r_{i+1}}, y_i$. If $\xi \not\in \bmu_{r_{i+1}}$ then $P_{\xi,i}(0,0) \neq 0$. Otherwise, i.e. if $\xi \in \bmu_{r_{i+1}}$, 
    the only vertices in the Newton polygon of $P_{\xi,i}$ are those corresponding to $x_i^{\kappa_{i+1}/r_{i+1}}$ and $y_i$.
\end{lemma}

\begin{proof}
    We proceed by induction on $i$, where the base case $i = 0$ is the trivial fact that for any $\xi \in \bmu_{r_1}$, the only vertices of the Newton polygon of $P_{\xi, 0}$ are the ones corresponding to $x_0^{\kappa_1/r_1}$ and $y_0$. Suppose that $\xi \in \bmu_{r_i}, 1 \leq i \leq g$. By the induction hypothesis $P_{\xi,0} = y_0 \cdots y_{i-2} P_{\xi,i-1}$, where $P_{\xi,i-1}$ is a power series in $x_{i-1}^{1/r_i}, y_{i-1}$ and 
    the only vertices in the Newton polygon of $P_{\xi,i-1}$ are those corresponding to $x_{i-1}^{\kappa_i/r_i}, y_{i-1}$.
 
    From \eqref{next-rupture-component} it follows that
    \begin{equation} \label{common-factor}
    \dfrac{x_{i-1}^{\kappa_i/r_i}}{y_{i-1}} = \dfrac{x_i^{\hat{r}_i \kappa_i / r_i}(A_i + B_i x_i + y_i)^{a_i \kappa_i / r_i}}{x_i^{\hat{\kappa}_i}(A_i + B_i x_i + y_i)^{b_i}} = (A_i + B_i x_i + y_i)^{(-1)^{n_i} r_{i+1}/r_i},
    \end{equation}
    which is a power series in $x_i^{1/r_{i+1}}, y_i$ by the binomial theorem. 
    By the condition on the vertices of the Newton polygon of $P_{\xi,i-1}$, any monomial that appears in $P_{\xi,i-1}$ with a nonzero coefficient is a multiple of either $x_{i-1}^{\kappa_i/r_i}$ or $y_{i-1}$. Therefore we may divide $P_{\xi,i-1}$ by $y_{i-1}$ and conclude that $P_{\xi,i-1} = y_{i-1} P_{\xi,i}$ with $P_{\xi,i}$ a power series in $x_i^{1/r_{i+1}}, y_i$.

    Consider the parametrization $(x_i(\tau), y_i(\tau))$ of the strict transform of $C^{[j]}$ at the moment of the resolution in which it intersects $R_i$, which results from lifting the original parametrization, $(x_0(\tau), y_0(\tau)) = (\tau^{r_1}, \varphi(\tau))$, through the blow-ups \eqref{next-rupture-component}. Let $P_{\xi,i}(\tau)$ be the result of substituting that parametrization into $P_{\xi,i}$, see Notation \ref{arc-notation}. Then,
    \[
    P_{\xi, 0}(\tau) = \varphi(\tau) - \varphi(\xi \tau) = \sum_{k \geq k_1} (1 - \xi^k) \varphi_k \tau^k.
    \]
    Note that $\bmu_{r_i} = \bmu_{r_1} \cap \bmu_{k_1} \cap \cdots \cap \bmu_{k_{i-1}}$. If $\xi \not\in \bmu_{r_{i+1}}$ then $\ord_\tau P_{\xi,0}(\tau) = k_i$. It follows from the resolution algorithm \cite[p. 512-516]{brieskorn1986} that $\ord_\tau x_j(\tau) = r_{j+1}, \ \ord_\tau y_j(\tau) = \kappa_{j+1}$ for all $j = 0, \ldots, g$. One can also see this by noting that $\ord_\tau x_g(\tau) = 1$ (because the strict transform $\mu_*^{-1}(C^{[j]})$ is smooth and transverse to $R_g = \{x_g = 0\}$) and then applying \eqref{next-rupture-component} repeatedly. Thus
    \[
    \ord_\tau P_{\xi,i}(\tau) = \ord_\tau P_{\xi,0}(\tau) - \sum_{j=0}^{i-1} \ord_\tau y_j(\tau) = k_i - (\kappa_1 + \cdots + \kappa_i) = 0,
    \]
    meaning $P_{\xi,i}$ has nonzero constant term. On the other hand, if $\xi \in \bmu_{r_{i+1}}$ then $\ord_\tau P_{\xi, 0} \geq k_{i+1}$, and therefore $\ord_\tau P_{\xi, i} \geq k_{i+1} - (\kappa_1 + \cdots + \kappa_i) = \kappa_{i+1}$. 
    Note that 
    \[
    \ord_\tau (x_i^{\alpha/r_{i+1}} y_i^\beta)(\tau) = \alpha + \beta \kappa_{i+1},
    \]
    so the only monomials with $\ord_\tau < \kappa_{i+1}$ are those with $\beta = 0$ and $0 \leq \alpha < \kappa_{i+1}$, and each of them has a different order in $\tau$. We conclude that the coefficient of each of these monomials in $P_{\xi, i}$ is necessarily zero. Therefore, the last thing that we need to check is that $x_i^{\kappa_{i+1}/r_{i+1}}$ and $y_i$ appear with nonzero coefficient in $P_{\xi, i}$. 

    For each $0 \leq i' \leq i$, let $c_{\xi, i'}(\alpha, \beta)$ be the coefficient of the monomial $x_{i'}^{\alpha/r_{{i'}+1}} y_{i'}^\beta$ in $P_{\xi,{i'}}$. Applying \eqref{next-rupture-component} to each monomial of $P_{\xi,{i'}-1} / y_{{i'}-1}$, we obtain
    \begin{equation} \label{change-coord-monomial}
    \dfrac{x_{{i'}-1}^{\alpha/r_{i'}} y_{{i'}-1}^{\beta}}{y_{{i'}-1}} = x_{i'}^{\frac{\alpha}{r_{i'}} \hr_{i'} + (\beta - 1) \hkappa_{i'}} (A_{i'} + B_{i'} x_{i'} + y_{i'})^{\frac{\alpha}{r_{i'}} a_{i'} + (\beta - 1) b_{i'}},
    \end{equation}
    which is a power series in $x_{i'}^{1/r_{{i'}+1}}, y_{i'}$. 
    
    For each $\widetilde{\alpha}, \widetilde{\beta} \in \Z_{\geq 0}$, let $I_{(\widetilde{\alpha}, \widetilde{\beta})} \subset \Z^2_{\geq 0}$ be the set of exponents $(\alpha, \beta)$ such that the monomial $x_{i'}^{\widetilde{\alpha}/r_{{i'}+1}} y_{i'}^{\widetilde{\beta}}$ appears with nonzero coefficient on the right hand side of \eqref{change-coord-monomial}. The coefficient $c_{\xi, {i'}}(\widetilde{\alpha}, \widetilde{\beta})$ is hence a linear combination of the coefficients $c_{\xi, {i'}}(\alpha, \beta)$ such that $(\alpha, \beta) \in I_{(\widetilde{\alpha}, \widetilde{\beta})}$.

    Recall that $(A_{i'} + B_{i'} x_{i'} + y_{i'})^{\frac{\alpha}{r_{i'}} a_{i'} + (\beta - 1) b_{i'}}$ is a power series with positive integer exponents by the binomial theorem, i.e. its monomials are of the form $x_{{i'}}^n y_{{i'}}^m$, $n, m \in \Z_{\geq 0}$. Therefore, a necessary condition so that $(\alpha, \beta) \in I_{(\widetilde{\alpha}, \widetilde{\beta})}$ is that there exist $n, m \in \Z_{\geq 0}$ solving the following system of equations:
    
    \begin{equation} \label{exponent-equations}
    \begin{cases}
        \widetilde{\alpha} = \alpha + (\beta - 1) \kappa_{i'} + n r_{{i'}+1} \\
        \widetilde{\beta} = m
    \end{cases}
    \end{equation}

    \medskip \noindent
    \textit{Claim.} For every integer ${i'} + 1 \leq \ell \leq g$, if $r_\ell$ does not divide $\widetilde{\alpha}$ and $\widetilde{\alpha} < k_\ell - k_{i'}$, then $c_{\xi, {i'}}(\widetilde{\alpha}, \widetilde{\beta}) = 0$ for every $\widetilde{\beta} \geq 0$.

    \medskip \noindent
    \textit{Proof of the claim.} This is an easy induction on ${i'}$. The case ${i'} = 0$ follows from the definitions of $P_{\xi,0}$ and $k_\ell$, see Definition \ref{characteristic-exponents}. Now let $(\alpha, \beta) \in I_{(\widetilde{\alpha}, \widetilde{\beta})}$, and since it satisfies \eqref{exponent-equations}, we obtain: 
    \begin{itemize}
        \item Because $\ell \geq {i'}+1$, we know that $r_\ell$ divides $\kappa_{i'}$ and $r_{{i'}+1}$. By assumption $r_\ell$ does not divide $\widetilde{\alpha}$, therefore $r_\ell$ does not divide $\alpha$ either.

        \item Because $\widetilde{\alpha} < k_\ell - k_{i'}$, we have
        \[
        \alpha < k_\ell - k_{i'} - (\beta - 1)\kappa_{i'} - nr_{{i'}+1} = k_\ell - k_{{i'}-1} - \beta \kappa_{i'} - n r_{{i'}+1} \leq k_\ell - k_{{i'}-1}.
        \]
    \end{itemize}
    Under these conditions, the inductive hypothesis says that $c_{\xi, {i'}-1}(\alpha, \beta) = 0$. Since this holds for every $(\alpha, \beta) \in I_{(\widetilde{\alpha}, \widetilde{\beta})}$, and $c_{\xi, {i'}}(\widetilde{\alpha}, \widetilde{\beta})$ is a linear combination of the coefficients $c_{\xi, {i'}}(\alpha, \beta)$ such that $(\alpha, \beta) \in I_{(\widetilde{\alpha}, \widetilde{\beta})}$, we deduce that $c_{\xi, {i'}}(\widetilde{\alpha}, \widetilde{\beta}) = 0$, concluding the induction on ${i'}$.

    Now let us study the coefficients $c_{\xi,{i'}}(k_\ell - k_{i'}, 0)$. Consider the equations \eqref{exponent-equations} with $\widetilde{\alpha} = k_\ell - k_{i'}$ and $\widetilde{\beta} = 0$. We find one solution
    \[
    \alpha = k_\ell - k_{i'} + \kappa_{i'} = k_\ell - k_{{i'}-1}, \quad \beta = 0, \quad n = 0, \quad m = 0.
    \]
    Any other solution has $\alpha < k_\ell - k_{{i'}-1}$, and since $r_\ell$ does not divide $k_\ell$, it does not divide $\alpha$ either. By the claim above, this means that all $(\alpha, \beta) \in I_{(k_\ell - k_{i'}, 0)}$ have $c_{\xi,{i'}-1}(\alpha, \beta) = 0$, except possibly $(\alpha, \beta) = (k_\ell - k_{{i'}-1}, 0)$. Hence,
    \begin{equation} \label{recurrence-coef-x}
    c_{\xi,{i'}}(k_\ell - k_{i'}, 0) = c_{\xi,{i'}-1}(k_\ell - k_{{i'}-1}, 0) A_{i'}^{a_{i'}(k_\ell - k_{{i'}-1})/r_{i'} - b_{i'}}.
    \end{equation}
    
    Consider now the coefficients $c_{\xi,{i'}}(0,1)$. The equations \eqref{exponent-equations} with $\widetilde{\alpha} = 0, \ \widetilde{\beta} = 1$ have the following solutions:
    \begin{enumerate}[label=(\roman{*})]
        \item $\alpha = \kappa_{i'}, \quad \beta = 0, \quad n = 0, \quad m = 1,$
        \item $\alpha = 0, \quad \beta = 1, \quad n = 0, \quad m = 1,$
        \item $\alpha = \kappa_{i'} - \nu r_{{i'}+1}, \quad \beta = 0, \quad n = \nu, \quad m = 1, \quad$ where $0 < \nu \leq \hkappa_{i'}$.
    \end{enumerate}
    The solutions of the form (iii) have $c_{\xi, {i'}-1}(\alpha, \beta) = 0$ by the inductive hypothesis on the Newton polygon of $P_{\xi,{i'}-1}$. Solution (ii) is not valid because it makes the exponent of the right term zero:
    \[
    (A_{i'} + B_{i'} x_{i'} + y_{i'})^{\frac{\alpha}{r_{i'}} a_{i'} + (\beta - 1) b_{i'}} = (A_{i'} + B_{i'} x_{i'} + y_{i'})^0 = 1,
    \]
    and therefore the right hand side of \eqref{change-coord-monomial} does not contain the monomial $y_{i'}$. Hence we are left with solution (i) alone. Recalling that $\kappa_{i'} = k_{i'} - k_{{i'}-1}$ we have
    \begin{equation} \label{recurrence-coef-y}
        c_{\xi,{i'}}(0,1) = c_{\xi,{i'}-1}(k_{i'} - k_{{i'}-1}, 0) \dfrac{a_{i'} \kappa_{i'}}{r_{i'}} A_{i'}^{(-1)^{n_{i'}} \tfrac{r_{{i'}+1}}{r_{i'}} - 1}. 
    \end{equation}
    Finally, observe that
    \begin{equation} \label{recurrence-base-case}
    c_{\xi, 0}(k_n, 0) = -\varphi_{k_n} \xi^{k_n} \neq 0, \quad c_{\xi,0}(0, 1) = 1 \neq 0.
    \end{equation}
    Putting together \eqref{recurrence-coef-x}, \eqref{recurrence-coef-y}, \eqref{recurrence-base-case} and the fact that $A_{i'} \neq 0$ for every ${i'}$, we arrive at the desired result that the coefficients of $x_i^{\kappa_{i+1}/r_{i+1}}$ and $y_i$ in $P_{\xi,i}$ are nonzero, i.e. $c_{\xi,i}(k_{i+1} - k_i, 0) \neq 0$ and $c_{\xi_i}(0,1) \neq 0$. This concludes the induction on $i$ and the proof.
\end{proof}

We end the section by giving formulas for the resolution data in terms of the Newton pairs.

\begin{proposition} \label{log-discrepancies}
    Let $E \in \mathcal{E}$ be an exceptional divisor with Newton pairs $(\hkappa_i^E, \hr_i^E), i = 1,\ldots,g^E$. Then
    \[
    \nu_E = k_{g^E}^E + r_1^E.
    \]
    
\end{proposition}

\begin{proof}
    Let $X \to \C^2$ be a composition of blow-ups and suppose $E_{(1,0)}$ and $E_{(0,1)}$ are two divisors on $X$ intersecting transversely at a point $Q \in X$, as in Proposition \ref{blowup-formula}, with log discrepancies $\nu_{(1,0)}$ and $\nu_{(0,1)}$ respectively. We claim that the log discrepancy of $E_{(\kappa,r)}$ is
    \[
    \nu_{(\kappa,r)} = \kappa \nu_{(1,0)} + r \nu_{(0,1)}.
    \]
    We prove this by induction on the pair $(\kappa,r)$. The formula is clearly true for the base cases $E_{(1,0)}$ and $E_{(0,1)}$. Now suppose that the formula holds true for divisors $E_{(\kappa',r')}$ and $E_{(\kappa'',r'')}$ over $Q$ intersecting transversely at a point. We are going to prove that the formula also holds for $E_{(\kappa,r)} = E_{(\kappa'+\kappa'',r'+r'')}$, the divisor that appears when we blow up the point $E_{(\kappa',r')} \cap E_{(\kappa'',r'')}$. Let $W \to X \to \C^2$ be the minimal composition of blow-ups that makes $E_{(\kappa',r')}$ and $E_{(\kappa'',r'')}$ appear (i.e. such that the centers of the valuations associated to $E_{(\kappa',r')}$ and $E_{(\kappa'',r'')}$ are divisors), and let $\sigma: \widetilde{W} \to W$ be the blow-up of the intersection point of $E_{(\kappa',r')}$ and $E_{(\kappa'',r'')}$. Then,
    \begin{align*}
    \ord_{E_{(\kappa,r)}} K_{\widetilde{W}/\C^2} &= \ord_{E_{(\kappa,r)}} (\sigma^* K_{W/\C^2} + E_{(\kappa,r)}) \\  
    &= \ord_{E_{(\kappa',r')}} K_{W/\C^2} + \ord_{E_{(\kappa'',r'')}} K_{W/\C^2} + 1
    \end{align*}
    and hence
    \[
    \nu_{E_{(\kappa,r)}} = \nu_{E_{(\kappa',r')}} + \nu_{E_{(\kappa'',r'')}} = (\kappa'+\kappa'')\nu_{(1,0)} + (r'+r'') \nu_{(0,1)} = \kappa \nu_{(1,0)} + r \nu_{(0,1)}.
    \]

    Applying this formula we see that
    \[
    \nu_{R_i^E} = \hkappa_i^E + \hr_i^E \nu_{R_{i-1}^E} \quad \text{for all } i = 1, \ldots, g^E,
    \]
    since $R_i^E$ is the divisor corresponding to the pair $(\hkappa_i^E, \hr_i^E)$ between a non-exceptional curve (which therefore has log discrepancy $1$) and $R_{i-1}^E$. In the case $i = 1$, a non-exceptional curve plays the role of ``$R_0^E$''. The result follows by recalling that $E = R_{g^E}^E$ and expanding the inductive formula.
\end{proof}

\begin{definition} \label{omega}
    Given $j = 1,\ldots,b$ and $E \in \mathcal{E}$, we denote by $\i = \i(E,j)$ the maximal integer $0 \leq i \leq g^E - 1$ such that $Q_i^E = Q_i^{[j]}$. Moreover, we say that $E$ and $C^{[j]}$ are \emph{compatible} if one of the following conditions is satisfied 
    \begin{enumerate}[label=(\roman{*})]
        \item $\kappa_{\i+1}^{[j]} / r_{\i+1}^{[j]} < 1$, or
        \item $\kappa_{\i+1}^{[j]} / r_{\i+1}^{[j]} > 1$ and the axes $L_\i^{[j]}, L_\i^E$ can be taken tangent to each other (recall that there is some freedom in the choice of the axes).
    \end{enumerate}
    Note that $E$ and $C^{[j]}$ are compatible if and only if the axis $L_\i^{[j]}$ can be chosen so that $R_{\i+1}^E$ is a divisor between $L_\i^{[j]}$ and $L_\i^{\prime[j]}$ in the sense of Definition \ref{divisors-between}.
\end{definition}

\begin{example}
\begin{enumerate}[label=(\roman{*}), wide]
    \item Consider the setting of Figure \ref{fig:res-algorithm-step-after}, ignoring the vertices in gray. For every branch $C^{[j]}, j = 1,\ldots,5$ and every $E$ in the picture we have $\i(E,j) = i - 1$, where $i$ is the step of the iteration depicted in the figure. Indeed, $Q_{i-1}^E = Q_{i-1}^{[j]} = Q_{i-1}$, where $Q_{i-1}$ is the point in Figure \ref{fig:res-algorithm-step-before}. 
    
    Every divisor in the picture is compatible with $C^{[5]}$. Indeed, $\kappa_i^{[5]} / r_i^{[5]} < 1$ because the strict transform is left of the divisor $E_{(1,1)}$. 
    
    On the other hand, $\kappa_i^{[1]} / r_i^{[1]} > 1$ and therefore only the divisors in the middle row are compatible with $C^{[1]}$. Indeed, the divisors in the top row have $L_i^E = L_i^{[3]} \neq L_i^{[1]}$, so they are not compatible with $C^{[1]}$, and similarly for the divisors in the bottom row.

    \item In Figure \ref{fig:example-resolution}, every divisor is compatible with $C^{[1]}$ and $C^{[2]}$ (note that $Q_1^{[1]} \neq Q_1^{[2]} = Q_1^{[3]} = Q_1^{[4]}$). The only divisors \emph{not} compatible with $C^{[3]}$ are the three divisors in the row of $\widetilde{C}^{[4]}$. Similarly, the only divisors \emph{not} compatible with $C^{[4]}$ are the three divisors in the row of $\widetilde{C}^{[3]}$.
\end{enumerate}
\end{example}

\begin{proposition} \label{multiplicities}
    Let $j = 1,\ldots,b$ and let $E \in \mathcal{E}$ be an exceptional divisor with Newton pairs $(\hkappa_i^E, \hr_i^E), i = 1,\ldots,g^E$. If $E$ and $C^{[j]}$ are compatible, then
    \[
    N_E^{[j]} = \kappa_1^E r_1^{[j]} + \cdots + \kappa_{\i}^E r_{\i}^{[j]} + \min \left\{\kappa_{\i+1}^{[j]} r_{\i+1}^E, \kappa_{\i+1}^E r_{\i+1}^{[j]} \right\}.
    \]
    Otherwise, if $E$ and $C^{[j]}$ are \emph{not} compatible, then
    \[
    N_E^{[j]} = \kappa_1^E r_1^{[j]} + \cdots + \kappa_{\i}^E r_{\i}^{[j]} + r_{\i+1}^E r_{\i+1}^{[j]}.
    \]
\end{proposition}


\begin{proof}
    Let $D$ be a curve in $\C^2$ whose lift $\widetilde{D}$ to the resolution $\mu$ intersects the divisor $E$ transversely at a point different from the intersection points of $E$ with other exceptional divisors or with the strict transform of $C$. We may express $N_E^{[j]}$ as an intersection number:
    \[
    N_E^{[j]} = N_E^{[j]} (E \cdot \widetilde{D}) = \mu^*(C^{[j]}) \cdot \widetilde{D} = C^{[j]} \cdot D.  
    \]
    Recall that this intersection product is computed by substituting a parametrization $(\delta_1(t), \delta_2(t))$ of $D$ into the equation of $C^{[j]}$, and computing the order of the resulting power series:
    \[
    C^{[j]} \cdot D = \ord_t f^{[j]}(\delta_1(t), \delta_2(t)).
    \]
    Our strategy to do this computation is to use the Puiseux decomposition \eqref{puiseux-decomposition}, substitute the parametrization of $D$ into each factor $P_{\xi,0} = y_0 - \varphi(\xi x_0^{1/r_1})$ separately, and finally compute the order using Lemma \ref{pxi-decomposition}. 
    
    The expressions $P_{\xi,0}$ are power series with fractional exponents, so in order to substitute the parametrization of $D$ into them one needs to choose a root $\delta_1(t)^{1/r_1}$. This root always exists as an element of $\C \llbracket t^{1/r_1} \rrbracket$, so the resulting orders $\ord_t P_{\xi,0}(\delta_1(t), \delta_2(t))$ are elements of $\Z[\tfrac{1}{r_1}]$. The idea of using fractional arcs is not new, see for instance \cite{denef2002} and \cite{pepereira2013}. Either way, we still have the equality
    \[
    \ord_t f^{[j]}(\delta_1(t), \delta_2(t)) = \sum_{\xi \in \bmu_{r_1^{[j]}}} \ord_t P^{[j]}_{\xi,0}(\delta_1(t), \delta_2(t)).
    \]
    

    Let us return to using Notation \ref{arc-notation} again, i.e. $P^{[j]}_{\xi,0}(t) \coloneqq P^{[j]}_{\xi,0}(\delta_1(t), \delta_2(t))$. Since $\widetilde{D}$ is transversal to $E$ (and by definition of the coordinates $(x_{g^E}^E, y_{g^E}^E)$), we have
    \[
    \ord_t x_{g^E}^E (t) = 1, \quad \ord_t y_{g^E}^E (t) = 0.
    \]
    Applying first \eqref{last-blow-up-E} and then \eqref{next-rupture-component} repeatedly, we find
    \[
    \ord_t x_i^E (t) = r_{i+1}^E, \quad \ord_t y_i^E (t) = \kappa_{i+1}^E, \quad \text{for all } i = 0, \ldots, g^E - 1.
    \]
    
    Now substitute these into the Puiseux series decomposition using Lemma \ref{pxi-decomposition}. By definition of $\i$, the coordinates $(x_i^{[j]}, y_i^{[j]})$ and $(x_i^E, y_i^E)$ coincide for all $i = 1, \ldots, \i - 1$, and they can be chosen to coincide for $i = \i$ if and only if $E$ is a divisor between $L_\i^{[j]}$ and $L_\i^{\prime[j]}$, i.e. if $E$ and $C^{[j]}$ are compatible. Hence if $\xi \in \bmu_{r_i^{[j]}} \setminus \bmu_{r_{i+1}^{[j]}}$ for some $i = 1, \ldots, \i$, then
    \[
    \ord_t P_{\xi,0}^{[j]}(t) = \ord_t y_0^E(t) + \cdots + \ord_t y_{i-1}^E(t) + \ord_t P_{\xi,i}^{[j]}(t) = \kappa_1^E + \cdots + \kappa_i^E + 0 = k_i^E.
    \]
    On the other hand, if $\xi \in \bmu_{r_{\i+1}^{[j]}}$, then the only vertices in the Newton polygon of $P_{\xi,\i}^{[j]}$ are those corresponding to $(x_{\i}^{[j]})^{\kappa_{\i+1}^{[j]}/r_{\i+1}^{[j]}}$ and $y_{\i}^{[j]}$ by Lemma \ref{pxi-decomposition}. Thus,
    \[
    \ord_t P_{\xi,\i}^{[j]}(t) = \min \left\{ \dfrac{\kappa_{\i+1}^{[j]}}{r_{\i+1}^{[j]}} \ord_t x_{\i}^{[j]}(t), \ord_t y_{\i}^{[j]}(t) \right\}.
    \]
    Summing over all possible values of $\xi \in \bmu_{r_1^{[j]}}$ we obtain
    \begin{align*}
    N_E^{[j]} &= C^{[j]} \cdot D = \ord_t f^{[j]}(t) = \sum_{\xi \in \bmu_{r_1^{[j]}}} \ord_t P_{\xi,0}^{[j]}(t) = \\ 
    &= (r_1^{[j]} - r_2^{[j]}) k_1^E + \cdots + (r_{\i}^{[j]} - r_{\i+1})^{[j]} k_{\i}^E + r_{\i+1}^{[j]} \left( k_{\i}^E + \ord_t P_{\xi,\i}^{[j]}(t) \right) = \\ 
    &= \kappa_1^E r_1^{[j]} + \cdots + \kappa_{\i}^E r_{\i}^{[j]} + \min \left\{ \kappa_{\i+1}^{[j]} \ord_t x_{\i}^{[j]}(t), r_{\i+1}^{[j]} \ord_t y_{\i}^{[j]}(t) \right\}.
    \end{align*}
    If $E$ and $C^{[j]}$ are compatible, then
    \[
    \ord_t x_{\i}^{[j]}(t) = \ord_t x_{\i}^E(t) = r_{\i+1}^E, \quad \ord_t y_{\i}^{[j]}(t) = \ord_t y_{\i}^E(t) = \kappa_{\i+1}^E.
    \]
    Otherwise, if $E$ and $C^{[j]}$ are not compatible, we must have $\hkappa_{\i+1}^E > \hr_{\i+1}^E$ and $\hkappa_{\i+1}^{[j]} > \hr_{\i+1}^{[j]}$. Hence
    \[
    \ord_t x_{\i}^{[j]}(t) = \ord_t y_{\i}^{[j]}(t) = \min \left\{\kappa_{\i+1}^E, r_{\i+1}^E \right\} = r_{\i+1}^E.
    \]
\end{proof}

\begin{remark} \label{special-case-multiplicities}
    At the end of Definition \ref{parts-resolution-graph}.\ref{special-case} we made the following claim: suppose all branches $C^{[j]}$ share the same divisor $R \coloneqq R_1^{[j]}$ (and hence there are two ends coming out of $R$ by Lemma \ref{end-well-defined}), and suppose both ends $\mathcal{F}_{R,1}$ and $\mathcal{F}_{R,2}$ contain at least one $m$-divisor each, then $R$ is an $m$-divisor. We can prove this using Proposition \ref{multiplicities}.
    
    Indeed, note that for every divisor $E \in \mathcal{F}_{R,1} \cup \mathcal{F}_{R,2}$ and every $j = 1, \ldots, b$ we have $\i(E,j) = 0$, and $E$ and $C^{[j]}$ are compatible. Thus by Proposition \ref{multiplicities} we have $N_E^{[j]} = \min\{\kappa_1^{[j]} r_1^E, \kappa_1^E r_1^{[j]}\}$.
    Also note that by assumption $\hkappa_1^{[j]} = \hkappa^R$ and $\hr_1^{[j]} = \hr^R$ for every $j = 1, \ldots, b$. Thus, if we denote $s \coloneqq \sum_{j=1}^b r_2^{[j]}$, then $N_E = s\hkappa^R r_1^E$ for every $E \in \mathcal{F}_{R,1}$, whereas $N_E = s\hr^R \kappa_1^E$ for every $E \in \mathcal{F}_{R,2}$. Thus, if there are $m$-divisors on both $\mathcal{F}_{R,1}$ and $\mathcal{F}_{R,2}$, then $N_R = s\hkappa^R \hr^R$ divides $m$, as we wanted.
\end{remark}


\section{Topology of the contact loci} \label{section:contact-loci}

We begin by recalling the basic definitions and results concerning contact loci. We restrict the definitions to our setting of plane curves; for the general definitions and basic properties of contact loci the reader may consult \cite{ein2004}.

\begin{definition} $\phantom{a}$
\begin{enumerate}[(i)]
    \item For any integer $m \geq 0$, an \emph{$m$-jet} of $\C^2$ is a morphism of $\C$-schemes $\gamma: \Spec \C[t]/(t^{m+1}) \to \C^2$. There exists a $\C$-scheme $\mathcal{L}_m(\C^2)$ parametrizing $m$-jets of $\C^2$, i.e. whose $\C$-points are
    \[
    \mathcal{L}_m(\C^2)(\C) = \mathrm{Mor}_\C(\Spec \C[t]/(t^{m+1}),\C^2).
    \]
    The scheme $\mathcal{L}_m(\C^2)$ is known as the \emph{$m$-jet space} of $\C^2$.

    \item An \emph{arc} of $\C^2$ is a morphism of $\C$-schemes $\gamma: \Spec \C \llbracket t \rrbracket \to \C^2$. There exists a $\C$-scheme $\mathcal{L}_\infty(\C^2)$ parametrizing arcs of $\C^2$, i.e. whose $\C$-points are
    \[
    \mathcal{L}_\infty(\C^2)(\C) = \mathrm{Mor}_\C(\Spec \C\llbracket t \rrbracket,\C^2).
    \]
    The scheme $\mathcal{L}_\infty(\C^2)$ is known as the \emph{arc space} of $\C^2$.

\end{enumerate}
There is a natural \emph{truncation map} $\pi_m: \mathcal{L}_\infty(\C^2) \to \mathcal{L}_m(\C^2)$, which on $\C$-points is given by
\[
\left(\sum_{i=0}^\infty a_it^i, \sum_{i=0}^\infty b_it^i\right) \longmapsto \left(\sum_{i=0}^m a_it^i, \sum_{i=0}^m b_it^i\right).
\]
\end{definition}

Note that an $m$-jet of $\C^2$ is equivalent to a pair $(\gamma_1(t),\gamma_2(t))$ of elements in $\C[t]/(t^{m+1})$. We will stress this relation by writing $\gamma(t) = (\gamma_1(t),\gamma_2(t))$. In turn, a truncated polynomial
\[
\sum_{i=0}^m a_it^i \in \C[t]/(t^{m+1})
\]
is determined by the tuple $(a_0,\dots,a_m)$. Hence $\mathcal{L}_m(\C^2)$ is isomorphic to $\C^{2(m+1)}$ as a $\C$-scheme. A similar discussion applies for the arc space $\mathcal{L}_\infty(\C^2)$.

Since $\mathcal{L}_m(\C^2) = \C^{2(m+1)}$ is a scheme of finite type over $\C$, we may consider both the Zariski and the analytic topologies on it. On the other hand, we will only consider $\mathcal{L}_\infty(\C^2)$ with its Zariski topology.

We will often consider the base point of an $m$-jet or arc $\gamma=(\gamma_1,\gamma_2)$ of $\C^2$, i.e. the point $\gamma(0) = (\gamma_1(0),\gamma_2(0)) \in \C^2$ obtained by evaluating at $t=0$.


In this paper we will only work with the $\C$-points of the $m$-jet space and the arc space of $\C^2$, so we will abuse notation and will also write $\mathcal{L}_m(\C^2)$ and $\mathcal{L}_\infty(\C^2)$ for the sets of their $\C$-points.

\begin{definition} \label{contact-locus-def}
    Let $f \in \C \llbracket x, y \rrbracket$ and let $m$ be a positive integer. The \emph{restricted $m$-contact locus} of $f$ at the origin is the space
    \[
    \X_m \coloneqq \X_{m}(f, 0) \coloneqq \{\gamma \in \mathcal{L}_m(\C^2) \mid \gamma(0) = 0,\ f(\gamma(t)) = t^m \mod t^{m+1}\}.
    \]
    Its lift to the arc space will be denoted by $\X^\infty_m = \pi_m^{-1}(\X_m)$. As with jet spaces and the arc space, the subset $\X_m$ can carry either the analytic or the Zariski topologies, whereas $\X^\infty_m$ will only be considered with its Zariski topology. The coefficient of the leading term of a nonzero power series $\psi \in \C \llbracket t \rrbracket$ is called the \emph{angular component}, and denoted by $\ac(\psi)$. The condition for an $m$-jet or an arc $\gamma$ to be in the restricted $m$-contact locus splits in three:
    \[
    \gamma(0) = 0, \quad \ord_t f(\gamma(t)) = m, \quad \ac f(\gamma(t)) = 1.
    \]
\end{definition}

\begin{definition}
    Let $f \in \C \llbracket x, y \rrbracket$ and let $m$ be a positive integer. Let $\mu: (Y, \mathbf{E}) \to (\C^2, 0)$ be the minimal $m$-separating log resolution of $C = V(f)$ and recall that we are denoting by $\mathcal{E}$ the set of irreducible components of the exceptional divisor $\mathbf{E}$. Every arc $\gamma \in \X_m^\infty$ can be lifted to a unique arc $\widetilde\gamma: \Spec \C\llbracket t \rrbracket \to Y$ in the resolution by the valuative criterion of properness.
    \begin{enumerate}[(i)]
        \item  For each exceptional component $E \in \mathcal{E}$ we define $\X_{m,E}^\infty$ to be the set of arcs in $\X_m^\infty$ whose lift to the resolution intersects $E$ and no other divisor, that is,
        \[
        \X_{m,E}^\infty \coloneqq \left\{\gamma \in \X_m^\infty \ \middle| \ \widetilde\gamma(0) \in E^\circ \right\},
        \]
        where $E^\circ \coloneqq E \setminus \cup_{E \neq F \in \mathcal{S} \cup \mathcal{E}} F$.

        \item For each rupture divisor $R \in \mathcal{R}$ we define $\frakZ_{m,R}^\infty$ to be the set of arcs in $\X_m^\infty$ that lift to the end $\mathcal{F}_R$ associated to $R$, see Definition \ref{parts-resolution-graph}. That is,
        \[
        \frakZ_{m,R}^\infty \coloneqq \bigcup_{E \in \mathcal{F}_R} \X_{m,E}^\infty.
        \]
    \end{enumerate}
\end{definition}

Note that the $m$-separating condition on $\mu$ already implies that for $\gamma \in \X_m^{\infty}$, the point $\widetilde\gamma(0)$ cannot lie in the intersection of two or more exceptional components. Furthermore, $\X_{m,E}^\infty$ is nonempty if and only if $E$ is an $m$-divisor, i.e. if $N_E$ divides $m$. 

\subsection{Decomposing the contact locus}

Even though each $\frakZ_{m,R}^\infty$ is a union of sets of the form $\X_{m,E}^\infty$ and these are set-theoretically disjoint, its topology (as a subspace of $\X_m^\infty$) is not the disjoint union topology. In fact, it may happen that for two different divisors $E, F \in \mathcal{E}$ we have $\overline{\X_{m,E}^\infty} \cap \overline{\X_{m,F}^\infty} \neq \varnothing$, in which case we say that arcs can \emph{jump} between $E$ and $F$. This is a typical phenomenon that happens even in the simplest cases.

\begin{example}
    Let $f(x,y)=y^2-x^3$ and $m=6$. The minimal log resolution $\mu: Y \to \C^2$ of $f$ is 6-separating, and it is obtained by a sequence of three blow-ups. We denote the exceptional divisors of $\mu$ by $E_1,E_2,E_3$ so that $E_i$ is the exceptional divisor of the $i$-th blow-up. One can easily check that
    
    \begin{center}
    $\begin{aligned}
        \X^\infty_{6,E_1} & = \{a_1 = b_1 = b_2 = 0, && a_2 = 0, \ b_3^2 = 1\}, \\
        \X^\infty_{6,E_2} & = \{a_1 = b_1 = b_2 = 0, && b_3 = 0, \ -a_2^3 = 1\}, \\
        \X^\infty_{6,E_3} & = \{a_1 = b_1 = b_2 = 0, && a_2 \neq 0, \ b_3 \neq 0, \ b_3^2-a_2^3 = 1\}.
    \end{aligned}$
    \end{center}
    We see that the arc $\gamma(t) = (0,t^3)$ is in $\X^\infty_{6,E_1}$, and also in the Zariski-closure of $\X^\infty_{6,E_3}$ in $\X^\infty_6$. Similarly, the arc $\eta(t) = (-t^2,0)$ is in $\X^\infty_{6,E_2}$, and also in the Zariski-closure of $\X^\infty_{6,E_3}$ in $\X^\infty_6$.
\end{example}

The following result gives a decomposition of $\X_m^\infty$ into disjoint closed subsets, and it says precisely which jumps are possible. It already appeared in \cite[Theorem 1.21]{budur2022nash} for the case where the curve $C$ is irreducible.

\begin{theorem} \label{arc-contact-loci-decomposition}
    Let $f \in \C \llbracket x,y \rrbracket$ be a reduced plane curve with minimal $m$-separating log resolution $\mu: (Y, \mathbf{E}) \to (\C^2, 0)$ whose dual graph we split into parts as in Definition \ref{parts-resolution-graph}. Then $\X_m^\infty$ decomposes as a Zariski-topologically disjoint union
    \[
    \X_m^\infty = \left( \bigsqcup_{R \in \mathcal{R}} \frakZ_{m,R}^\infty \right) \sqcup \left( \bigsqcup_{R,S \in \mathcal{R} \cup \mathcal{S}} \bigsqcup_{E \in \mathcal{T}_{R,S}} \X_{m,E}^\infty  \right).
    \]
    Moreover, each $\frakZ_{m,R}^\infty$ is nonempty if and only if there are $m$-divisors in the end $\mathcal{F}_R$, and in that case $\frakZ_{m,R}^\infty = \overline{\X_{m,E_\dom(m,R)}^\infty}$, where the closure is taken inside $\X_m^\infty$.
\end{theorem}

\begin{proof}
    
    
    Let $R' \in \mathcal{R}$ and $E \in \mathcal{T}_{R,S}$ for some $R,S \in \mathcal{R} \cup \mathcal{S}$. By \cite[Theorem A]{ein2004}, the subsets $\X_m^\infty$, $\mathfrak{Z}_{m,R'}^\infty$ and $\X_{m,E}^\infty$ are cylinders, i.e. there exists $\ell \gg 0$ such that $\X_m^\ell \coloneqq \pi_\ell(\X^\infty_m)$, $\mathfrak{Z}_{m,R'}^\ell \coloneqq \pi_\ell(\mathfrak{Z}_{m,R'}^\infty)$ and $\X_{m,E}^\ell \coloneqq \pi_\ell(\X_{m,E}^\infty)$ are constructible subsets of $\mathcal{L}_m(\C^2)$, and $\X_m^\infty = \pi_\ell^{-1}(\X_m^\ell)$, $\mathfrak{Z}_{m,R'}^\infty = \pi_\ell^{-1}(\mathfrak{Z}_{m,R'}^\ell)$ and $\X_{m,E}^\infty = \pi_\ell^{-1}(\X_{m,E}^\ell)$. Therefore, it is enough to show that the subsets $\mathfrak{Z}_{m,R'}^\ell$ and $\X_{m,E}^\ell$ are Zariski-closed in $\X^\ell_m$. In turn, since they are constructible, it suffices to check they are closed in the analytic topology of $\X_{m}^\ell$.

    Let $E, F \in \mathcal{E}$ be $m$-divisors, and suppose that $\X^\ell_{m,E} \cap \overline{\X^\ell_{m,F}} \neq \varnothing$, where the closure is taken in $\X^\ell_m$. By the curve selection lemma \cite[Lemma 3.1]{milnor1968}, there exists a complex analytic map $(\C,0) \to \X^\ell_m: s \mapsto \gamma_s^\ell$, such that $\gamma^\ell_0 \in \X^\ell_{m,E}$ and $\gamma^\ell_s \in \X^\ell_{m,F}$ for $s \neq 0$. Composing it with the section
    \[
    \mathcal{L}_\ell(\C^2) \longrightarrow \mathcal{L}_\infty(\C^2): \quad \left(\sum_{i=0}^\ell a_it^i, \sum_{i=0}^\ell b_it^i\right) \longmapsto \left(\sum_{i=0}^\ell a_it^i, \sum_{i=0}^\ell b_it^i\right),
    \]
    i.e. lifting each jet to an arc by considering every coefficient of $t^k$ for $k > \ell$ to be zero, we get an analytic family of arcs $\gamma_s$ such that $\gamma_0 \in \X_{m,E}$ and $\gamma_s \in \X_{m,F}$ for every $s \neq 0$. Note that
    \[
    m = \gamma_s \cdot C = \sum_{j = 1}^b \gamma_s \cdot C^{[j]} \quad \text{for all } s \in \C.
    \]
    For each $j = 1,\ldots,b$, the value $m^{[j]} \coloneqq \gamma_s \cdot C^{[j]}$ must be constant for all $s \in \C$ because of the upper semicontinuity of the intersection product. Therefore the family of arcs $\gamma_s$ lies in the $m^{[j]}$-th contact locus of $f^{[j]}$ for each $j = 1,\ldots,b$.

    Suppose that $E$ and $F$ do not lie in a common end $\mathcal{F}_R$. In that case there exists a $j = 1,\ldots,b$ such that the centers of $E$ and $F$ in the minimal $m^{[j]}$-separating log resolution of $C^{[j]}$ lie in different trunks or ends, or in different divisors of the same trunk. The results in \cite[\S 7.3]{budur2022nash} show that it is not possible for arcs in the $m^{[j]}$-th contact locus of the irreducible curve $f^{[j]}$ to jump between such divisors (i.e. between divisors in different trunks or ends, or divisors in the same trunk). 
    Hence the family $\gamma_s$ cannot exist, and it is not possible to jump between $E$ and $F$.

    To finish the proof we just need to show that all arcs in $\X_m^\infty$ that lift to an end $\mathcal{F}_R$ can be reached jumping from $E_\dom(m,R)$, i.e. that $\frakZ_{m,R}^\infty = \overline{\X_{m,E_\dom(m,R)}^\infty}$. To do this we are once again going to reduce the situation to the irreducible case, which was also proved in \cite[\S 7.3]{budur2022nash}. The idea here is to use Proposition \ref{multiplicities} to show that the intersection numbers $\gamma \cdot C^{[j]}$ are constant for $\gamma \in \frakZ_{m,R}^\infty$ and deduce that $\frakZ_{m,R}^\infty$ is precisely the set of arcs that lift to an end in the minimal resolution of $C^{[j]}$. Details are as follows. 
    
    Fix $\mathcal{F}_R$ and observe that the integer $N_E^{[j]}/\hr^E$ is independent of $E$ for all $E \in \mathcal{F}_R$. Indeed, for each branch $C^{[j]}$ and every $E \in \mathcal{F}_R$, either they are compatible in the sense of Definition \ref{omega}, or $\kappa^{[j]}_{\i+1} / r^{[j]}_{\i+1} \leq \kappa^E_{\i+1} / r^E_{\i+1}$. In either case, $\hr^E$ appears as a common factor in the formulas of Proposition \ref{multiplicities}, so $N_E^{[j]}/\hr^E$ is an integer. To see the independence, note that
    \[
    \kappa_i^E/\hr^E = \hkappa_i^E \hr_{i+1}^E \cdots \hr_{g^E-1}^E, \quad r_i^E/\hr^E = \hr_i^E \cdots \hr_{g^E-1}^E \quad \text{for every } i = 1, \ldots, g^E
    \]
    are independent of $E \in \mathcal{F}_R$, because all divisors in $\mathcal{F}_R$ share all Newton pairs except the last one. Applying this to the formulas in Proposition \ref{multiplicities} shows that $N_E^{[j]}/\hr^E$ is independent of the divisor $E$ for $E \in \mathcal{F}$.
    
    Since this independence holds for all $j = 1,\ldots,b$, the integer $N_E/\hr^E$ is also independent of $E$ for $E \in \mathcal{F}_R$. Let $\gamma \in \frakZ_{m,R}^\infty$ be an arc lifting to $E \in \mathcal{F}_R$. Then
    \[
    m^{[j]} \coloneqq \gamma \cdot C^{[j]} = N_E^{[j]} (\widetilde\gamma \cdot E) = \dfrac{N_E^{[j]}}{N_E} N_E (\widetilde\gamma \cdot E) = \dfrac{N_E^{[j]}}{N_E} (\gamma \cdot C) = \dfrac{N_E^{[j]}}{N_E} m = \dfrac{N_E^{[j]}/\hr^E}{N_E/\hr^E} m
    \]
    is independent of $E \in \mathcal{F}_R$. Note that there is at least one $j = 1,\ldots,b$ such that $R$ is also a rupture divisor in the minimal log resolution $\mu^{[j]}$ of $f^{[j]}$. Since $m^{[j]}$ is independent of $E \in \mathcal{F}_R$, we see that $\frakZ_{m,R}^\infty$ is precisely the set of arcs in the $m^{[j]}$-th contact locus of $f^{[j]}$ that lift to the end associated to $R$ in $\mu^{[j]}$. Therefore the problem reduces to the case of the irreducible curve $C^{[j]}$. Once again, the irreducible case was already proved in \cite[\S 7.3]{budur2022nash} --- the idea of the proof is to use approximate roots to reduce the situation to a curve with one Puiseux pair, and then apply \cite[Lemma 3.11]{ishii2008} for the divisorial valuations, which are now toric.
\end{proof}

\begin{remark}
    The value $N_E/\hr^E$ is in general \emph{not} independent of $E$ for $E \in \mathcal{T}_{R,S}$. Indeed, the argument that we used for $\mathcal{F}_R$ does not work for divisors in a trunk, since in that case there is at least one $j$ for which $\hkappa^E$ appears in the formula for $N_E^{[j]}$ instead of $\hr^E$.
\end{remark}

\begin{figure}[ht] 
    \centering
    \resizebox{.8\linewidth}{!}{%
    \fontsize{20pt}{20pt}\selectfont
    \def\svgheight{0.5cm}
    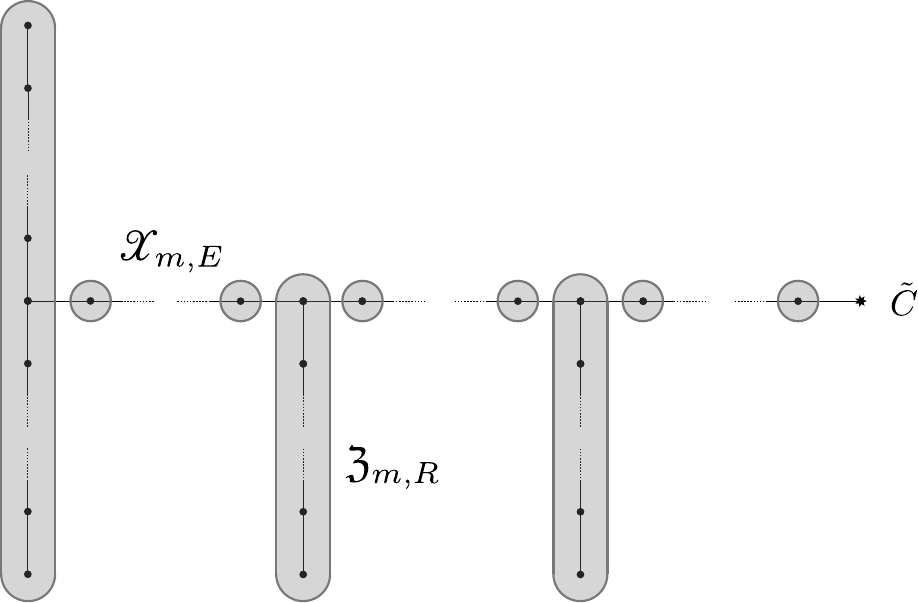
}
    \caption{Resolution graph of an $m$-separating log resolution of an irreducible curve. All arcs lifting to divisors in the same gray region belong to the same component of the decomposition of $\X_m^\infty$ given in Theorem \ref{arc-contact-loci-decomposition}.}
    \label{fig:resolution-graph}
\end{figure}

This theorem provides information about the topology of the lift of the contact loci to the arc space. To bring the result back down to the jet space, note that $\X_m = \pi_m(\X_m^\infty) = \pi_m(\pi_m^{-1}(\X_m))$, so the connected components of $\X_m$ are the images of the connected components of $\X_m^\infty$ and in particular the topological disjoint unions of Theorem \ref{arc-contact-loci-decomposition} are preserved.

\begin{corollary} \label{contact-loci-decomposition}
    Define $\X_{m,E} \coloneqq \pi_m(\X_{m,E}^\infty),\ \frakZ_{m,R} \coloneqq \pi_m(\frakZ_{m,R}^{\infty})$. Then the $m$-contact locus $\X_m^\infty$ of $f$ decomposes as a Zariski-topologically disjoint union
    \[
    \X_m = \left( \bigsqcup_{R \in \mathcal{R}} \frakZ_{m,R} \right) \sqcup \left( \bigsqcup_{R,S \in \mathcal{R} \cup \mathcal{S}} \bigsqcup_{E \in \mathcal{T}_{R,S}} \X_{m,E} \right).
    \]
    Moreover, each $\frakZ_{m,R}$ is nonempty if and only if there are $m$-divisors in the end $\mathcal{F}_R$, and in that case $\frakZ_{m,R} = \overline{\X_{m,E_\dom(m,R)}}$, where the closure is taken inside $\X_m$.
\end{corollary}

Therefore in order to compute the topology of the contact locus $\X_m$ we only need to understand the topology of each of the sets in the decomposition. We are going to give explicit equations for each of these sets, and finally we will use the formula
\begin{equation} \label{XmE-dimension}
    \dim_{\C} \X_{m,E} = 2(m + 1) - m_E \nu_E - 1
\end{equation}
to count the number of free variables, see \cite[Lemma 2.6]{budur2020cohomology}.

\begin{lemma} \label{ac-equations}
    Let $\gamma(t)$ be an arc in $\C^2$ such that $\ord_t f(\gamma(t)) = m$, and whose lift to the resolution $\mu$ intersects the exceptional divisor $E \in \mathcal{E}$, with Newton pairs $(\hkappa_i^E, \hr_i^E)$, $i = 1,\ldots,g^E$, see Definition \ref{newton-pairs-E}. In the $(x_{g^E}^E, y_{g^E}^E)$ coordinates, see \eqref{last-blow-up-E}, this lift is of the form
    \[
    x_{g^E}^E(t) = \alpha t^{m_E} + O(t^{m_E+1}), \quad y_{g^E}^E(t) = \beta + O(t),
    \]
    where $\alpha \in \C^\times$ and $\beta \in \C$ avoids the values that would make $\widetilde\gamma$ intersect other exceptional divisors or strict transforms. Then, for each branch $C^{[j]}$ and depending on the positioning of $C^{[j]}$ with respect to $E$ in the sense of Definition \ref{left-right-up}, the leading coefficient of $f^{[j]}(t)$ is
    \[
    \ac f^{[j]}(t) = \begin{cases}
        \star \alpha^{N_E^{[j]}} \beta^{N_E^{[j]} a^E/\hr^E} & \text{if $\widetilde{C}^{[j]}$ lies to the left of $E$}, \\
        \star \alpha^{N_E^{[j]}} \beta^{a^E (S_\omega/\hr^E) + \, b^E r_{\omega+1}^{[j]}} \phantom{\left( \beta \right)^{r_{\i+2}^{[j]}}} & \text{if $\widetilde{C}^{[j]}$ lies to the right of $E$}, \\
        \star \alpha^{N_E^{[j]}} \beta^{a^E (S_\omega/\hr^E) + \, b^E r_{\omega+1}^{[j]}} \left(z_u + \beta \right)^{r_{\i+2}^{[j]}} & \text{if $\widetilde{C}^{[j]}$ lies in $u \in U^E$ above $E$.}
    \end{cases}
    \]
    Here the stars $\star$ represent nonzero complex numbers which do not depend on $\alpha, \beta$; we define $S_\omega \coloneqq \kappa_1^E r_1^{[j]} + \cdots + \kappa_\omega^E r_\omega^{[j]}$; and $z_u$ is the $y$-coordinate at which the divisor in $u \in U^E$ meets $E$. 
\end{lemma}

\begin{proof}
    The computation is essentially the same as the one in the proof of Proposition \ref{multiplicities}, but keeping track of the angular components instead of the orders. For a cleaner notation, let $(b^E, a^E) \coloneqq (b_{g^E}^E, a_{g^E}^E)$. Using \eqref{last-blow-up-E} and \eqref{next-rupture-component} we get
    \[
    \ac x_{g^E-1}^E(t) = \alpha^{\hr^E} \beta^{a^E}, \quad \ac y_{g^E-1}^E(t) = \alpha^{\hkappa^E} \beta^{b^E},
    \]
    \[
    \ac x_i^E(t) = \star \alpha^{r_{i+1}^E} \beta^{r_{i+1}^E a^E / \hr^E}, \quad \ac y_i^E(t) = \star \alpha^{\kappa_{i+1}^E} \beta^{\kappa_{i+1}^E a^E/\hr^E}, \quad \text{for all } i = 1, \ldots, g^E - 2.
    \]
    The stars denote some products of powers of the constants $A_i^E$, which are nonzero and do not depend on $\alpha, \beta$. For each $j = 1,\ldots,b$ we know the vertices of the Newton polygon of $P_{\xi,\i}^{[j]}$ by Lemma \ref{pxi-decomposition}, and hence we know that
    \[
    \ac P_{\xi,\i}^{[j]}(t) = \ac \left(\star x_\i^{[j]}(t)^{\kappa_{\i+1}^{[j]}/r_{\i+1}^{[j]}} + \star y_\i^{[j]}(t)\right).
    \]
    The angular component of a sum of two power series is the angular component of the one with smaller order, or the sum of angular components if both of them have the same order. Note that which of them has a smallest order will depend on
    \begin{enumerate*}[(i)]
        \item whether $\i < g^E - 1$ or $\i = g^E - 1$; 
        \item whether $E$ and $C^{[j]}$ are compatible or not; and
        \item whether $\kappa_{\i+1}^{[j]}/r_{\i+1}^{[j]}$ is bigger, smaller or equal to $\kappa_{\i+1}^E/r_{\i+1}^E$.
    \end{enumerate*}
    In fact, if $E$ and $C^{[j]}$ are \emph{not} compatible or $\i < g^E - 1$, we see that $C^{[j]}$ is to the left of $E$ (essentially because the path from $E$ to $\widetilde{C}^{[j]}$ in the dual graph starts by going in the direction of the root). Otherwise, if $E$ and $C^{[j]}$ are compatible and $\i = g^E - 1$, the remaining condition (iii) gives the different possibilities:
    \begin{enumerate}[label=(\arabic{*}), wide, labelindent=0pt]
        \item If $\kappa_{\i+1}^{[j]}/r_{\i+1}^{[j]} < \kappa_{\i+1}^E/r_{\i+1}^E$, then $\widetilde C^{[j]}$ is to the left of $E$.

        \item If $\kappa_{\i+1}^{[j]}/r_{\i+1}^{[j]} = \kappa_{\i+1}^E/r_{\i+1}^E$, then $\widetilde C^{[j]}$ is above $E$.

        \item If $\kappa_{\i+1}^{[j]}/r_{\i+1}^{[j]} > \kappa_{\i+1}^E/r_{\i+1}^E$, then $\widetilde C^{[j]}$ is to the right of $E$.

    \end{enumerate}

    Since the full proof would involve doing the same computation over and over again, we only show here how to tackle the case where $E \in \mathcal{R}_0$ and $\widetilde{C}^{[j]}$ lies above $E$, which contains the difficulties that appear in all other cases. This corresponds to the case where
    \begin{enumerate*}[(i)]
        \item $\i = g^E - 1$,
        \item $E$ and $C^{[j]}$ are compatible, and
        \item $\kappa_{\i+1}^{[j]}/r_{\i+1}^{[j]} = \kappa_{\i+1}^E/r_{\i+1}^E$.
    \end{enumerate*}
    Note that we have $(x_i^{[j]}, y_i^{[j]}) = (x_i^E, y_i^E)$ for all $i = 1, \ldots, \i$ and also
    \[
    (x_{\i+1}^{[j]}, y_{\i+1}^{[j]}) = (x_{\i+1}^E, y_{\i+1}^E + z_u + \star x_{\i+1}^E)
    \]
    where $z_u$ is the $y$-coordinate of the point in $E$ that has to be blown up to continue the resolution of $C^{[j]}$ and $\star$ is a complex number that we do not care about. Therefore $\ord_t x_{\i+1}^{[j]}(t) > 0$ and $\ord_t y_{\i+1}^{[j]}(t) = 0$, and thus
    \[
    \ac P_{\xi,\i+1}^{[j]}(t) = \star \ac y_{\i+1}^{[j]}(t) = \star (z_u + \beta).
    \]
    Let $F$ be the divisor between $L_{g^E - 1}^E$ and $L_{g^E - 1}^{\prime E}$ corresponding to the pair $(b^E, a^E)$. Then we have $\kappa_i^F = \kappa_i^E a^E / \hr^E$ for all $i = 1,\ldots,\i$ and $\kappa_{\i+1}^F = b^E$. Using the decomposition given in Lemma \ref{pxi-decomposition} and the formula of Proposition \ref{multiplicities}, we compute
    \begin{align*}
        \ac f^{[j]}(t) &= \star \ac \left( y_0^{[j]}(t)^{r_1^{[j]}} \cdots y_{\i-1}^{[j]}(t)^{r_\i^{[j]}} y_\i^{[j]}(t)^{r_{\i+1}^{[j]}} \prod_{\xi \in \bmu_{r_{\i+2}^{[j]}}} P_{\xi,\i+1}^{[j]}(t) \right) = \\ 
        &= \star \alpha^{\kappa_1^E r_1^{[j]} + \cdots + \kappa_\i^E r_\i^{[j]} + \kappa_{\i+1}^E r_{\i+1}^{[j]}} \beta^{\frac{a^E}{\hr^E} \left(\kappa_1^E r_1^{[j]} + \cdots + \kappa_\i^E r_\i^{[j]}\right) + b^E r_{\i+1}^{[j]}} (z_u + \beta)^{r_{\i+2}^{[j]}} = \\ 
        &= \star \alpha^{N_E^{[j]}} \beta^{a^E (S_\omega/\hr^E) + \, b^E r_{\omega+1}^{[j]}} (z_u + \beta)^{r_{\i+2}^{[j]}},
    \end{align*}
    as we wanted to show. The other cases are similar.
\end{proof}

\begin{theorem} \label{XmE-computation}
    Let $E$ be an exceptional divisor in the trunk $\mathcal{T}_{R,S}$ for some $R, S \in \mathcal{R} \cup \mathcal{S}$. Then,
    \[
    \X_{m,E} \cong \bigsqcup^{c(\mathcal{T}_{R,S})} \C^\times \times \C^{2m - m_E \nu_E}.
    \]
\end{theorem}

\begin{proof}
    Let $\mu_\infty: \mathcal{L}_\infty(Y) \to \mathcal{L}_\infty(\C^2)$ be the map induced on arcs by the resolution. Note that $\mu_\infty^{-1}(\X_{m,E}^\infty)$ is the set of arcs $\widetilde\gamma(t)$ in the $(x_{g^E}^E, y_{g^E}^E)$ plane such that
    \[
    x_{g^E}^E(t) = \alpha t^{m_E} + O(t^{m_E+1}), \quad y_{g^E}^E(t) = \beta + O(t) \quad \text{and} \quad \ac (f \circ \mu)(t) = 1.
    \]
    Since $E$ is in a trunk, $E$ is not a rupture component. Therefore any $\widetilde{C}^{[j]}$ lies either to the left or to the right of $E$. In this situation, Lemma \ref{ac-equations} says that
    \[
    \ac (f \circ \mu)(t) = \prod_{j} \ac (f^{[j]} \circ \mu)(t) = \star \alpha^{\hr^E K_1 + \hkappa^E K_2} \beta^{a^E K_1 + b^E K_2},
    \]
    where $\star$ is a nonzero complex number, and the integers $K_1, K_2$ are given by the formulas
    \[
    K_1 = \sum_{j \in \mathcal{S}^E_\leftarrow} \dfrac{N_E^{[j]}}{\hr^E} + \sum_{j \in \mathcal{S}^E_\rightarrow} \dfrac{S_\omega}{\hr^E}, \quad K_2 = \sum_{j \in \mathcal{S}^E_\rightarrow} r_{\omega+1}^{[j]}.
    \]
    Here, we are using the fact that $N_E^{[j]} = S_\i + \hkappa^E r_{\i+1}^{[j]}$ for every $j \in S^E_\rightarrow$ (deduced from Proposition \ref{multiplicities}).
    
    By \eqref{kappa-r-b-a-matrix}, the $2 \times 2$ matrix with columns $(\hkappa^E, \hr^E)$ and $(b^E, a^E)$ is invertible over $\Z$, and therefore $\gcd(\hr^E K_1 + \hkappa^E K_2 , a^E K_1 + b^E K_2) = \gcd(K_1, K_2)$. On the other hand, let $F$ be a divisor such that $E \cap F \neq \varnothing$. Then the matrix with columns $(\hkappa^E, \hr^E)$ and $(\hkappa^F, \hr^F)$ is also invertible over $\Z$ (this follows by induction on the coprime pair $(\kappa, r)$, since either $F$ is the exceptional divisor of blowing up the intersection of $E$ with another divisor, or vice versa). Applying Proposition \ref{cE-independence}.\ref{cE-independence-trunk} we get
    \begin{equation} \label{K1-K2-gcd}
    c(\mathcal{T}_{R,S}) = c(E) = \gcd(N_E, N_F) = \gcd(\hr^E K_1 + \hkappa^E K_2, \hr^F K_1 + \hkappa^F K_2) = \gcd(K_1, K_2)
    \end{equation}
    and we conclude that the equation $\ac(f \circ \mu)(t) = 1$ defines $c(\mathcal{T}_{R,S})$ disjoint copies of $\C^\times$ in the $(\alpha, \beta)$-plane.
    

    Now we have to translate this result down from arcs to jets, and from the resolution to $\C^2$. The second proof of \cite[Lemma A.1]{budur2020cohomology} shows that, provided we look at jets with constant order of contact with an exceptional divisor, the map induced by a single blow-up on jets is just the truncation of some of the higher-order variables of one of the components. Since all $\widetilde\gamma \in \mu_\infty^{-1}(\X_{m,E}^\infty)$ intersect $E$ with order $m_E$, we may apply this result for all blow-ups of the resolution. Note that the only coefficients of $\widetilde\gamma$ that have to satisfy an equation are the initial ones, $\alpha$ and $\beta$, which never get truncated. The coefficients of the higher order terms are free. 

    More precisely, denoting by $\mu_m: \mathcal{L}_m(Y) \to \mathcal{L}_m(\C^2)$ the map induced by the resolution on jets, \textit{loc. cit.} says that $\mu_m^{-1}(\X_{m,E}) = \pi_m(\mu_\infty^{-1}(\X_{m,E}^\infty))$, and that the map $\mu_m$ just truncates some of the free variables. Therefore $\X_{m,E}$ is isomorphic to $\sqcup^{c(\mathcal{T}_{R,S})} \C^\times \times \C^M$ for some $M \in \N$. Now we just need to use formula \eqref{XmE-dimension} for the dimension of $\X_{m,E}$ to conclude that $M = 2(m+1) - m_E \nu_E - 2$ as we wanted.
\end{proof}

In the following theorem, we will denote by $\mathrm{val}(R)$ the valency of a rupture divisor $R \in \mathcal{R}$ in the dual graph of $\mu^\ast C$, i.e. the number of vertices connected to it.

%
%
%
%
%
%
%
%
%
%

\begin{theorem} \label{rupture-topology}
    Let $R \in \mathcal{R}$ be a rupture divisor.
\begin{enumerate}[(i)]
    \item If $R$ is an $m$-divisor, then
    \[
    \frakZ_{m,R} \cong \bigsqcup^{c(R)} T \times \C^{2m - m_R \nu_R},
    \]
    where $T$ is a non-compact Riemann surface of genus
    \[
    g(T) = 1+\frac{1}{2c(R)} \Bigg( N_R(\mathrm{val}(R)-2)-\sum_{\substack{E \in (\mathcal{E} \cup \mathcal{S})\setminus \{R\} \\ R \cap E \neq \varnothing}} \gcd(N_R,N_E) \Bigg)
    \]
    with
    $
    \frac{1}{c(R)}\sum_{S \in \mathcal{N}_R} \gcd(N_R,N_S)
    $
    punctures.

    \item If $R$ is not an $m$-divisor but there are $m$-divisors in $\mathcal{F}_R$, then
    \[
    \frakZ_{m,R} \cong \bigsqcup^{c(\mathcal{F}_R)} \C^{2m - m_{E_\dom(m,R)} \nu_{E_\dom(m,R)} + 1}.
    \]
    In the exceptional case in which $R$ has two ends attached, the disjoint union is taken over $c(\mathcal{F}_{R,i})$, where $i=1,2$ is such that $\mathcal{F}_{R,i}$ is the end containing $E_\dom(m,R)$.

    \item Otherwise, if there are no $m$-divisors in $\mathcal{F}_R$, then $\frakZ_{m,R} = \varnothing$.
\end{enumerate}
\end{theorem}

\begin{proof}
    By the same arguments as in the proof of Theorem \ref{XmE-computation}, for each $E \in \mathcal{F}_R$ the leading term $\ac(f \circ \mu)(t)$ is a polynomial in $\alpha, \beta$ that we can compute using Lemma \ref{ac-equations}, and the space $\X_{m,E}$ is isomorphic to $\{(\alpha, \beta) \in \C^2 \mid \ac(f \circ \mu)(t) = 1\} \times \C^M$. Using \eqref{XmE-dimension} we know $M = 2m - m_E \nu_E$, so all we need to do is study the curve given by $\{\ac(f \circ \mu)(t) = 1\}$.

    Since the computations for $R$ in $\mathcal{R}_0$ and $\mathcal{R}_1$ are completely analogous, let us focus in the case $R \in \mathcal{R}_0$. By Lemma \ref{ac-equations},
    \begin{equation} \label{ac-rupture-equation}
    \ac(f \circ \mu)(t) = \star \alpha^{\hr^R K_1 + \hkappa^R K_2} \beta^{a^R K_1 + b^R K_2} \prod_{u \in U^R} (z_u + \beta)^{\sum_{j \in \mathcal{S}^R_u} r_{\i+2}^{[j]}}.
    \end{equation}
    where $\star$ is a nonzero complex number and the integers $K_1, K_2$ are given by the formulas
    \[
    K_1 = \sum_{j \in \mathcal{S}^R_\leftarrow} \dfrac{N_R^{[j]}}{\hr^R} + \sum_{j \in \mathcal{S}^R_\rightarrow} \dfrac{S_\omega}{\hr^R} + \sum_{u \in U^R} \sum_{j \in \mathcal{S}^R_u} \dfrac{S_\omega}{\hr^R}, \quad K_2 = \sum_{j \in \mathcal{S}^R_\rightarrow} r_{\omega+1}^{[j]} + \sum_{u \in U^R} \sum_{j \in \mathcal{S}^R_u} \dfrac{S_\omega}{\hr^R}.
    \]
    The curve $\{\ac(f \circ \mu)(t) = 1\}$ is smooth (which can be checked by differentiating with respect to $\alpha$) and has as many connected components as the greatest common divisor of the exponents in the right-hand side of equation \eqref{ac-rupture-equation} above, which we compute now.

    For each $u \in U^R$ and each $j \in \mathcal{S}^R_u$, Proposition \ref{multiplicities} gives the following formula for divisors $E_{(\kappa, r)}$ between $L_{\i+1}^{[j]}$ and $L_{\i+1}^{\prime [j]} = R$:
    \[
    N_{E_{(\kappa, r)}}^{[j]} = r N_R^{[j]} + \kappa r_{\i+2}^{[j]}.
    \]
    In each connected component $u \in U^R$, there is precisely one divisor which is adjacent to $R$, call it $E_u$. Its corresponding coprime pair $(\kappa_u, r_u)$ must have $\kappa_u = 1$. Since we also know that for $j \not\in \mathcal{S}^R_u$ we have $N_{E_{(\kappa,r)}}^{[j]} = r N_R^{[j]}$, we conclude
    \[
    N_{E_u} = N_{E_{(\kappa_u, r_u)}} = \sum_{j = 1}^b N^{[j]}_{E_{(\kappa_u, r_u)}} = r_u N_R + \sum_{j \in \mathcal{S}_u^R} r_{\i+2}^{[j]},
    \]
    and therefore $\gcd(N_R, \sum_{j \in \mathcal{S}_u^{[j]}} r_{\i+2}^{[j]}) = \gcd(N_R, N_{E_u}) = c(\mathcal{T}_{R,S})$, where $S \in \mathcal{N}_R$ is the unique rupture divisor or strict transform which is a neighbor of $R$ and lies in the component $u \in U^R$. On the other hand, using Proposition \ref{multiplicities} to obtain the formula $N_R = \hr^R K_1 + \hkappa^R K_2$ and denoting by $E$ the only divisor to the right of $R$ such that $E \cap R \neq \varnothing$, we have by the same argument that we used to prove \eqref{K1-K2-gcd} that
    \begin{equation} \label{gcd-computation-1}
    \gcd(\hr^R K_1 + \hkappa^R K_2, a^R K_1 + b^R K_2) = \gcd(K_1, K_2) = \gcd(N_R, N_E).
    \end{equation}
    
    Therefore the gcd of the exponents in $\ac(f \circ \mu)(t)$ is equal to the gcd of $N_R$ and all the $N_E$, where $E$ runs through all divisors adjacent to $R$ \emph{except for one} (the divisor $F$ immediately left of $R$). Nevertheless, by Proposition \ref{cE-independence}.\ref{cE-simply} 
    removing one of the $N_E$ from the computation of the gcd does not change the result, and we conclude that $\{\ac(f \circ \mu)(t) = 1\}$ has $c(R)$ connected components.

    The equation of each connected component is simply
    \[
    \alpha^{N_R/c(R)} \beta^{(a^R K_1 + b^R K_2)/c(R)} \prod_{u \in U^R} (z_u + \beta)^{\left(\sum_{j \in \mathcal{S}^R_u} r_{\i+2}^{[j]}\right)/c(R)} = \zeta,
    \]
    where $\zeta$ runs over the $c(R)$-th roots of unity. Projecting to the $\beta$ coordinate we obtain a covering of degree $N_R/c(R)$ over $\C \setminus (\{0\} \cup \{-z_u \mid u \in U\})$, i.e. over $\P^1$ with $\mathrm{val}(R)$ points removed, from which we can obtain the genus using the Riemann-Hurwitz formula as follows. 

    The ramification index at the point $-z_u$ is the greatest common divisor of the exponents of $\alpha$ and $(z_u + \beta)$, which we have already shown to be $c(\mathcal{T}_{R,S})/c(R)$, where $S \in \mathcal{N}_R$ is the neighbor of $R$ living in the component $u \in U^R$. The ramification index at $0$ is the greatest common divisor of the exponents of $\alpha$ and $\beta$, which by \eqref{gcd-computation-1} is $\gcd(N_R, N_E)/c(R)$ where $E$ is the unique divisor adjacent to $R$ and to the right of $R$. Finally, the ramification index at infinity is the greatest common divisor of $N_R/c(R)$ and the sum of all other exponents, which is $\gcd(N_R, N_F)/c(R)$ where $F$ is the unique divisor adjacent to $R$ and to the \emph{left} of $R$ (once again, this is because the sum of all $N_{E'}$ where $E'$ runs through the divisors adjacent to $R$ is a multiple of $N_R$). This way we have counted the points that $\{\ac(f \circ \mu)(t) = 1\}$ is missing to be compact.

    Recall from Theorem \ref{contact-loci-decomposition} that $\frakZ_{m,R} = \overline{\X_{m,E_\dom(m,R)}}$. The effect of this closure in the topology is to ``fill in'' the missing points of $\{\ac(f \circ \mu)(t) = 1\}$ that lie above the point which corresponds to the intersection with the end $\mathcal{F}_R$. This can be seen for instance by repeating the computation of Lemma \ref{ac-equations} but considering the coefficients in the $(x_{g^R-1}^R, y_{g^R-1}^R)$. From here the result follows for $R \in \mathcal{R}_0$. The computation for $R \in \mathcal{R}_1$ is totally analogous.
\end{proof}

\section{Floer homology of the Milnor fibration} \label{section:floer-homology}

In this section assume $f$ is a reduced convergent power series. As explained in the introduction, the Milnor fiber $\F = M_f$ is naturally a Liouville domain, and moreover, the Milnor fibration admits a monodromy $\varphi: \F \to \F$ and a grading on $\varphi$ so that $(\F,\lambda,\varphi)$ is a graded abstract contact open book. The goal of this section is to compute the Floer homology groups $\HF_*(\varphi^m,+)$ for every $m \geq 1$.

The main tools we use to do so are the symplectic monodromy at radius zero of Fernández de Bobadilla and Pe{\l}ka \cite{bobadilla2022}, which enhances A'Campo's model of the Milnor fiber \cite{acampo1975} with a symplectic structure and a symplectic monodromy with good dynamical properties; and the McLean spectral sequence \cite[\S 4]{mclean2019}, \cite[Proposition 6.3]{bobadilla2022}, which makes use of those dynamical properties to compute the Floer homology. To study the degeneration properties of the McLean spectral sequence, we adapt the ideas of \cite{seidel1996} to our setting.

\subsection{Graded abstract contact open books and Floer homology}

For the convenience of the reader, we recall the main notions about fixed-point Floer homology for Liouville domains that will be used in the rest of the paper. We follow \cite{dostoglou-salamon94}, \cite{mclean2019}, \cite{seidel2001}, \cite{uljarevic2014}.

\subsubsection{Graded abstract contact open books}

A \emph{Liouville domain} is a pair $(M,\lambda)$, where $M$ is a compact connected manifold with boundary and $\lambda$ is a 1-form on $M$, such that $d\lambda$ is a symplectic form, and the Liouville vector field $X_\lambda$ (i.e. the vector field defined by $d\lambda(X_\lambda,\cdot) = \lambda$) points outwards on $\partial M$. We say that a diffeomorphism $\varphi: M \to M$ is \emph{exact} if $\varphi^\ast\lambda-\lambda$ is an exact 1-form on $M$. In that case, a function $\mathcal{a}: M \to \R$ such that $\varphi^\ast \lambda-\lambda = -d\mathcal{a}$ is called an \emph{action} of $\varphi$. It is unique up to an additive constant. A diffeomorphism $\varphi: M \to M$ is said to be \emph{compactly supported} if it equals the identity on a neighborhood of $\partial M$. An \emph{abstract contact open book} \cite[Definition 3.12]{mclean2019} is a tuple $(M,\lambda,\varphi)$ such that $(M,\lambda)$ is a Liouville domain, and $\varphi$ is a compactly supported exact symplectomorphism.

Let $(M,\omega)$ a symplectic manifold of dimension $2n$ and $\mathrm{Fr}(M)$ its symplectic frame bundle, see \cite[Definition 3.2]{mclean2019}. Recall that it is a $\mathrm{Sp}(2n)$-principal bundle. A \emph{grading} on a symplectic manifold $(M,\omega)$ is a $\widetilde{\mathrm{Sp}}(2n)$-principal bundle $\widetilde{F}$ together with an isomorphism $\widetilde{F} \times_{\widetilde{\mathrm{Sp}}(2n)} \mathrm{Sp}(2n) \to \mathrm{Fr}(M)$. Here, $\widetilde{\mathrm{Sp}}(2n)$ denotes the universal cover of $\mathrm{Sp}(2n)$. A symplectic manifold together with a grading is said to be \emph{graded}.

Assume that $M$ is graded and let $\varphi: M \to M$ be a symplectomorphism. Denote by $\widetilde{F}$ the grading on $M$. We have a natural isomorphism $\beta: \mathrm{Fr}(M) \to \varphi^\ast \mathrm{Fr}(M)$ of $\mathrm{Sp}(2n)$-principal bundles. A \emph{grading} on $\varphi$ is an isomorphism $\widetilde{\beta}: \widetilde{F} \to \varphi^\ast \widetilde{F}$ of $\widetilde{\mathrm{Sp}}(2n)$-principal bundles such that $\widetilde{\beta} \times_{\widetilde{\mathrm{Sp}}(2n)} \mathrm{Sp}(2n) = \beta$. A symplectomorphism together with a grading is said to be \emph{graded}.

Let $(M,\lambda,\varphi)$ be an abstract contact open book. A \emph{grading} on $(M,\lambda,\varphi)$ is a grading on $(M,d\lambda)$ and a compatible grading on $\varphi$. An an abstract contact open book together with a grading is said to be \emph{graded}. For more details on gradings, see \cite[App. A]{mclean2019}.

Let $M$ be a graded symplectic manifold of dimension $2n$ with grading $\widetilde{F}$, and let $\varphi: M \to M$ be a graded symplectomorphism with grading $\widetilde{\beta}: \widetilde{F} \to \varphi^\ast \widetilde{F}$. Let $p \in M$ be a fixed point of $\varphi$. The grading $\beta$ induces an $\widetilde{\mathrm{Sp}}(2n)$-equivariant isomorphism $\widetilde{\beta}_p: \widetilde{F}_p \to \widetilde{F}_p$. This corresponds to the action of an element in $\widetilde{\mathrm{Sp}}(2n)$, i.e. a homotopy class of paths in $\mathrm{Sp}(2n)$ starting at the identity matrix. The \emph{Conley-Zehnder} index of $p$ with respect to $\varphi$ is defined to be the Maslov index of this path, see \cite{robbin93}.


\subsubsection{Floer homology}

Now we assign a Floer homology group to each graded abstract contact open book $(M,\lambda,\varphi)$. Fix the following data:
\begin{itemize}
    \item A $\varphi$-periodic time-dependent Hamiltonian $H_t$; i.e. a smooth function $H: M \times \R \to \R: (x,t) \mapsto H_t(x)$ such that $H_t \circ \varphi = H_{t+1}$.
    \item A $\varphi$-periodic time-dependent cylindrical almost complex structure $J_t$; i.e. a smooth family of cylindrical almost complex structures such that $\varphi^\ast J_t = J_{t+1}$. For the definition of cylindrical almost complex structure see \cite[Definition 4.1]{mclean2019} or \cite[(2.7),(2.8)]{uljarevic2014}.
\end{itemize}

Let $X^H$ be the unique time-dependent vector field such that $\omega(X^H,\cdot) = dH_t$, and we denote its flow by $\psi^H_t$. A \emph{Hamiltonian $\varphi$-twisted loop} is an integral curve $\eta: \R \to M$ of the vector field $X^H$ such that $\varphi(\eta(t+1)) = \eta(t)$ for every $t \in \R$. Observe that then $\eta(0)$ is a fixed point of $\varphi \circ \psi^H_1$. We define the \emph{Conley-Zehnder} index of $\eta$ to be the Conley-Zehnder index of the point $\eta(0)$ with respect to $\varphi \circ \psi^H_1$.

Let $\eta_-$ and $\eta_+$ be two Hamiltonian $\varphi$-twisted loops. A \emph{Floer trajectory} from $\eta_-$ to $\eta_+$ is a smooth map $u: \R \to M$ such that
\begin{itemize}
    \item $\varphi(u(s,t+1)) = u(s,t)$,
    \item $\lim_{s \to \pm \infty} = \eta_{\pm}(t)$, and
    \item $u$ satisfies the Floer Cauchy-Riemann equation
    \[
    \frac{\partial u}{\partial s}+J_t(u)\left(\frac{\partial u}{\partial t}-X^H(u)\right) = 0.
    \]
\end{itemize}

Suppose that $(H_t,J_t)$ is a regular pair (see \cite[p. 589]{dostoglou-salamon94}, \cite[Definition 2.5]{uljarevic2014}). Given two Hamiltonian twisted loops $\eta_-$ and $\eta_+$, we denote by $\mathcal{M}(\eta_-,\eta_+)$ the moduli space of Floer trajectories connecting $\eta_-$ and $\eta_+$. The additive group $\R$ acts smoothly and freely on $\mathcal{M}(\eta_-,\eta_+)$ by reparametrization, i.e. if $a \in \R$ and $u(s,t) \in \mathcal{M}(\eta_-,\eta_+)$, then $u(s+a,t) \in \mathcal{M}(\eta_-,\eta_+)$. If $\CZ(\eta_+)-\CZ(\eta_-) = 1$, then the quotient $\mathcal{M}(\eta_-,\eta_+)/\R$ is a compact 0-dimensional manifold, see \cite[Theorem 2.17]{uljarevic2014}. Moreover, in \cite[\S 5]{floer93} Floer and Hofer give a way of orienting such moduli spaces in the case $\varphi = \id$. This extends to the present situation.

The \emph{Floer complex} $\mathrm{CF}_\bullet(H_t,J_t)$ is the free abelian group generated by the Hamiltonian $\varphi$-twisted loops, with grading given by the opposite of the Conley-Zehnder index, i.e. $\deg(\eta) \coloneqq -\CZ(\eta)$. The differential $\partial: \mathrm{CF}_\bullet(H_t,J_t) \to \mathrm{CF}_{\bullet-1}(H_t,J_t)$ is defined by
\[
\partial (\eta_-) = \sum n(\eta_-,\eta_+)\eta_+,
\]
where $n(\eta_-,\eta_+)$ is the \emph{signed} count of points of the compact 0-dimensional manifold $\mathcal{M}(\eta_-,\eta_+)/\R$, according to the given orientation.
By \cite[p. 589]{dostoglou-salamon94}, cf. \cite[Theorem 2.18]{uljarevic2014}, this is indeed a complex, i.e. $\partial^2 = 0$.

\begin{definition}
    The \emph{Floer homology} of the graded abstract contact open book $(M,\lambda,\varphi)$ is the homology of the chain complex $(\mathrm{CF}_\bullet(H_t,J_t),\partial)$. We denote it by $\HF_\bullet(\varphi,+)$.
\end{definition}

Fixed-point Floer homology deals with the dynamics of an abstract contact open book. The following notion will be used several times.

\begin{definition}[{\cite[Definition 5.25]{bobadilla2022}}] \label{def-czf}
    Let $(M,\lambda,\varphi)$ be an abstract contact open book. A compact connected subset $B$ of $M$ is said to be a \emph{codimension zero family of fixed points} if there exists an open neighborhood $N_B$ of $B$, and a smooth function $H_B: N_B \to \R$, such that
    \begin{enumerate}[(i)]
        \item $B$ is a codimension zero submanifold of $M$.
        \item $\mathrm{Fix}(\varphi|_{N_B}) = B$.
        \item $H_B^{-1}(0) = B$.
        \item $\varphi|_{N_B}$ is the time-one flow of the vector field $X^{H_B}$.
    \end{enumerate}
\end{definition}

\begin{remark}
    This definition is the one given in \cite{bobadilla2022}. It is broader than the one used in \cite[\S 4]{mclean2019}, where the Hamiltonian $H_B$ is required to be nonnegative.
\end{remark}

Given a codimension zero family of fixed points $B$, observe that $B \cap \partial M$ (which can be empty) is contained in $\partial B$. Given $p \in \partial B \setminus (B \cap \partial M)$, there is an open neighborhood $U$ of $p$ contained in $N_B$ such that $H_B|_{U \setminus \partial B} > 0$, or $H_B|_{U \setminus \partial B} < 0$. We denote by $\partial^+B$ the set of points of $\partial B \setminus (B \cap \partial M)$ such that $H_B|_{U \setminus \partial B} > 0$, and by $\partial^-B$ the set of points such that $H_B|_{U \setminus \partial B} < 0$. We will call $\partial^+B$ and $\partial^-B$ the \emph{positive} and \emph{negative} parts of the boundary, respectively.

Note that an action $\mathcal{a}$ of $\varphi$ is constant on $B$, and we will denote its value by $\mathcal{a}(B)$. Similarly, since two homotopic paths in $\mathrm{Sp}(2n)$ relative to their endpoints have the same Maslov index, the Conley-Zehnder index relative to $\varphi$ is constant on $B$. We will denote its value by $\CZ(B)$.

\subsection{The A'Campo representative of the monodromy}
We begin by recalling A'Campo's description of the monodromy in terms of an embedded resolution of singularities \cite{acampo1975} in the case of a plane curve, see also \cite[Ch. 9]{wall2004}. Fix a log resolution $\mu: (Y,\mathbf{E}) \to (\C^2,0)$ of $f$, which we will later on assume to be $m$-separating, see Section \ref{curves}. Recall that $\mathcal{E}$ is the set of irreducible components of $\mathbf{E}$, and $\mathcal{S}$ is the set of irreducible components of the strict transform $\widetilde{C} = \mu_*^{-1}(C)$. The Milnor fiber $\F$ admits a decomposition
\begin{equation} \label{fiber-decomposition}
\F = \bigcup_{E \in \mathcal{E} \cup \mathcal{S}} \F_E \cup \bigcup_{E,F \in \mathcal{E} \cup \mathcal{S}} \F_{E,F},
\end{equation}
where the pieces of the decomposition are described as follows.

For each $E \in \mathcal{E} \cup \mathcal{S}$ there is a \emph{branched} covering $\widetilde{E} \to E$ of degree $N_E$ whose branch locus are the intersection points of $E$ with other divisors in $\mathcal{E} \cup \mathcal{S}$, i.e. the covering branches at
$E \setminus E^\circ$ where $E^\circ \coloneqq E \setminus \cup_{E \neq F \in \mathcal{E} \cup \mathcal{S}} F$. The set $\F_E \subset \F$ is the real oriented blow-up of $\widetilde{E}$ at the preimages of the points in $E \setminus E^\circ$, i.e. $\F_E$ is the result of replacing each point of $\widetilde{E}$ lying over $E \setminus E^\circ$ by a copy of $\SS^1$.

Explicitly, the covering $\widetilde{E} \to E$ is constructed as follows, see \cite[Lemma 3.2.3]{denef1998}. For every $p \in E^\circ$, there exists an open neighborhood $V$ of $p$ with coordinates $(z_E,x)$ such that $(f \circ \mu)|_V = uz_E^{N_E}$, where $u$ is an invertible function on $V$, and $E \cap V = \{z_E = 0\}$. We consider
\[
\widetilde{V} \coloneqq \{(x,w) \in (E \cap V) \times \C \ | \ w^{N_E} = u(0,x)\},
\]
which comes with the natural projection $\widetilde{V} \to E \cap V$. This is a covering map of degree $N_E$.

On the other hand, if $p \in E \cap F$ for another divisor $F$, there exists an open neighborhood $V$ of $p$ with coordinates $(z_E,z_F)$ such that $f|_V = uz_E^{N_E}z_F^{N_F}$, where $u$ is an invertible function on $V$, $E \cap V = \{z_E = 0\}$, and $F \cap V = \{z_F = 0\}$. We consider $\widetilde{V}$ to be the normalization of
\[
\{(z_F,w) \in (E \cap V) \times \C \ | \ w^{N_E} = u(0,z_F)z_F^{N_F}\}.
\]
Composing the normalization and the natural projection, we get a map $\widetilde{V} \to E \cap V$. This is a branched covering of degree $N_E$; its only branch value is $0 \in V$ and it has index $\gcd(N_E, N_F)$.

Gluing the two types of coverings over an open cover of $E$ yields the desired branch covering $\widetilde{E} \to E$.

Note that for any $E, F \in \mathcal{E} \cup \mathcal{S}$ such that $E \cap F \neq \varnothing$, the spaces $\F_E$ and $\F_F$ have (possibly many) boundary components above the intersection point, and each of these components is isomorphic to $\SS^1$. The set $\F_{E,F}$ is a disjoint union of cylinders (i.e. copies of $[0,1] \times \SS^1$), each connecting one of the boundary components of $\F_E$ with one of the boundary components of $\F_F$.

For a proof of the decomposition \eqref{fiber-decomposition} see \cite{acampo1975} or \cite{bobadilla2022}. Figure \ref{fig:milnor-fiber} shows a schematic picture of the decomposition for the Milnor fiber of an irreducible plane curve. The following result gives the topology of each piece of the decomposition, using the description of the dual graph given in Definition \ref{parts-resolution-graph}.

\begin{proposition} \label{acampo-pieces}
    The topology of each $\F_E$ can be described as follows:
    \begin{enumerate}[(i)]
        \item For every $j = 1,\ldots,b$, the space $\F_{\widetilde{C}^{[j]}}$ is a cylinder.

        \item If $E \in \mathcal{E}$ has valency 1, then it lies in some end $\mathcal{F}_R$ (resp. $\mathcal{F}_{R,i}$ in the exceptional case, see Definition \ref{parts-resolution-graph}), and $\F_E$ is a disjoint union of $c(\mathcal{F}_R)$ (resp. $c(\mathcal{F}_{R,i})$) disks.

        \item If $E$ has valency 2, then it lies in either a trunk $\mathcal{T}_{R,S}$ or an end $\mathcal{F}_R$ (or $\mathcal{F}_{R,i}$ in the exceptional case), and $\F_E$ is a disjoint union of $c(\mathcal{T}_{R,S})$ or $c(\mathcal{F}_R)$ (or $c(\mathcal{F}_{R,i})$) cylinders, respectively.

        \item If $R$ has valency $\geq 3$, (that is, it is a rupture component $R \in \mathcal{R}$), then $\F_R$ is a disjoint union of $c(R)$ compact orientable surfaces, each of them with
        \[
        \frac{1}{c(R)} \sum_{\substack{E \in (\mathcal{E} \cup \mathcal{S})\setminus \{R\} \\ R \cap E \neq \varnothing}} \gcd(N_R,N_E)
        \]
        boundary components and genus
        \[
        1+\frac{1}{2c(R)} \Bigg( N_R(\mathrm{val}(R)-2)-\sum_{\substack{E \in (\mathcal{E} \cup \mathcal{S})\setminus \{R\} \\ R \cap E \neq \varnothing}} \gcd(N_R,N_E) \Bigg).
        \]

    \end{enumerate}
\end{proposition}

\begin{proof}
    Note that each strict transform $\widetilde{C}^{[j]}$ is homeomorphic to a closed disk, as it is the intersection of a smooth curve with a closed ball. Therefore $\F_{\widetilde{C}^{[j]}}$ is the real oriented blow-up of a covering of $(\widetilde{C}^{[j]})^\circ$ of degree $N_{\widetilde{C}^{[j]}} = 1$, and hence it is topologically a cylinder.

    Let $E \in \mathcal{E}$, we are going to see that $\widetilde{E}^\circ$ has $c(E)$ connected components. Since $\widetilde{E}^\circ \to E^\circ$ is a covering of degree $N_E$, we may identify its fibers with the group $\Z/(N_E)$. The fundamental group of $E^\circ \cong \mathbb{P}^1 \setminus \{\mathrm{val}(E) \text{ points}\}$ is generated by a loop around each of the missing points, subject to the relation that the composition of all those generators is trivial. Let us study the action of a loop around one of the missing points on the fiber.

    Let $(z_E, z_F)$ be local coordinates on a neighborhood $V$ of the intersection point of $E$ and $F$ in which $f(z_E,z_F) = z_E^{N_E} z_F^{N_F}$. By definition, the restriction of the covering to $E^\circ \cap V$ is given in coordinates by the projection
    \[
    \{(z_F, w) \in (E^\circ \cap V) \times \C \mid w^{N_E} = u(0, z_F) z_F^{N_F}\} \to E^\circ \cap V: (z_F, w) \mapsto z_F.
    \]
    Hence if we fix $\gamma: [0,1] \to E^\circ \cap V: t \mapsto \delta e^{2 \pi \imath t}$ a small loop around the origin and a preimage $(w, z_F)$ of $\gamma(0)$ as the start $\widetilde{\gamma}(0)$ of a lift, we have $\widetilde{\gamma}(1) = (w e^{2\pi \imath N_F/N_E}, z_F)$. In other words, under the identification of the fiber with $\Z/(N_E)$, the action of the loop $\gamma$ is a adding $N_F$. Therefore the orbits of the action of $\pi_1(E^\circ)$ on $\Z/(N_E)$ are in bijection with 
    \[
    \Z/(\langle N_E \rangle + \langle \{ N_F \mid F \in \mathcal{E} \setminus \{E\}, E \cap F \neq \varnothing \} \rangle ) = \Z / \langle c(E) \rangle,
    \]
    which shows that $\widetilde{E}^\circ$ has $c(E)$ connected components, as we wanted.

    From here, the result follows by recalling that $c(E) = c(\mathcal{T}_{R,S})$ (resp. $c(\mathcal{F}_R)$, $c(\mathcal{F}_{R,i})$) if $E \in \mathcal{T}_{R,S}$ (resp. $c(\mathcal{F}_R)$, $c(\mathcal{F}_{R,i})$), and using the Riemann-Hurwitz formula to compute the genus of each component of $\widetilde{E}^\circ$.

\end{proof}

\begin{figure}[ht] 
    \centering
    \resizebox{0.9\linewidth}{!}{%
    \fontsize{20pt}{20pt}\selectfont
    \def\svgheight{0.5cm}
    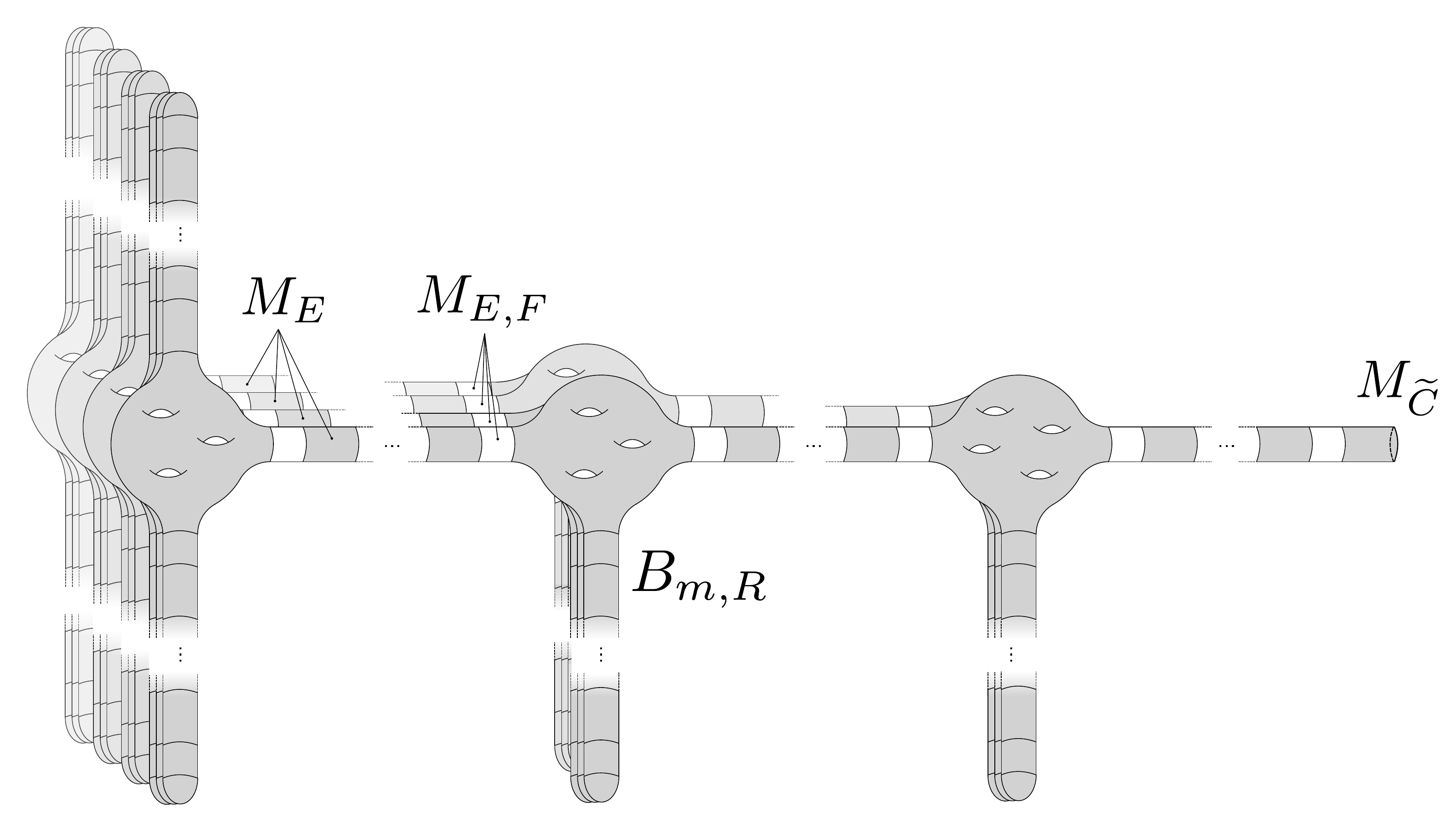
}
    \caption{Schematic picture of the partition of the Milnor fiber of an irreducible curve according to \eqref{fiber-decomposition}.}
    \label{fig:milnor-fiber}
\end{figure}

\begin{remark} \label{NR-divides-sum-of-NE}
    From the above, one can obtain a topological proof of Proposition \ref{cE-independence}. Indeed, consider the action of $\pi_1(E^\circ)$ on the fiber $\Z/(N_E)$ of the covering $\widetilde{E}^\circ \to E^\circ$. Since the composition of a loop around each of the missing points of $E^\circ$ is homotopic to a constant path, we conclude that the action of this loop, which is adding $N_F$ for every divisor $F$ that intersects $E$, is trivial. This means that $\sum_{F \in \mathcal{E} \cup \mathcal{S} \setminus \{E\},\, E \cap F \neq \varnothing} N_F$ is a multiple of $N_E$.
\end{remark}

A'Campo's construction comes with a canonical representative of the monodromy $\varphi: \F \to \F$. Its iterates $\varphi^m$ acts on each $\F_E$ as the covering transformations. That is, locally around each intersection $E \cap F$, the maps $\varphi^m,\ m \geq 1$, cyclically permute the boundary components of $\F_E$ that intersect $\F_{E,F}$. When $\gcd(N_E,N_F)$ divides $m$, each boundary component is mapped to itself rotated by a certain angle. The $m$-th iterate of the monodromy action induces a cyclic permutation on the set of connecting tubes $\F_{E,F}$, and when $\gcd(N_E, N_F)$ divides $m$ each tube is mapped to itself with a twist that interpolates between the rotations of the boundary components. 

\begin{remark}
The twist interpolating between the rotations may be described as follows. For each connected component $K$ of $\F_{E,F}$ there exists a diffeomorphism $(v,\theta): K \to [0,1] \times \SS^1$ such that $v^{-1}(0) = K \cap \F_E,\ v^{-1}(1) = K \cap \F_F$ and such that, if $\gcd(N_E, N_F)$ divides $m$, then the map $\varphi^m$ in these coordinates is
\begin{equation} \label{monodromy-tube-formula}
(v,\theta) \mapsto \left(v, \theta + \dfrac{ma_F}{N_E} (1 - v) - \dfrac{ma_E}{N_F} v \right),
\end{equation}
where $a_E, a_F$ are the coefficients in the Bézout identity $\gcd(N_E,N_F) = a_E N_E + a_F N_F$. 

This follows from the original work by A'Campo \cite[p. 240]{acampo1975}. Nevertheless, we provide a proof using the coordinate description of the A'Campo space given by Bobadilla and Pe{\l}ka, see \cite[Ex. 4.16, Figure 4(b)]{bobadilla2022}.

\begin{proof}
Recall that on $\F_{E,F}$ we have angular coordinates $\theta_E, \theta_F \in \R/\Z$ satisfying $N_E \theta_E + N_F \theta_F = 0$, and weighted barycentric coordinates $\ell_E, \ell_F \in \R$ such that $N_E \ell_E + N_F \ell_F = 1$, see \cite[(119)]{bobadilla2022}. In these coordinates, the map $\varphi^m$ is given by $\theta_E \mapsto \theta_E + m \ell_E, \theta_F \mapsto \theta_F + m \ell_F$. Let $v \coloneqq N_F \ell_F$, so $1 - v = N_E \ell_E$, and let $\theta \in \R/\Z$ be a parameter for the connected component of the set $\{N_E \theta_E + N_F \theta_F = 0\} \subset (\R/\Z)^2$ containing the point $(0,0)$. Then we have
\[
\theta_E = \dfrac{N_F}{\gcd} \theta, \quad \theta_F = \dfrac{-N_E}{\gcd} \theta,
\]
where we are denoting $\gcd = \gcd(N_E, N_F)$ for brevity. Let $\theta'$ be the angular coordinate of $\varphi^m(v, \theta)$, then
\[
\dfrac{N_F}{\gcd} \theta + m \ell_E = \dfrac{N_F}{\gcd} \theta', \quad \dfrac{N_E}{\gcd} \theta + m \ell_F = \dfrac{N_E}{\gcd} \theta'
\]
and therefore
\begin{align*}
    0 &= \dfrac{a_F N_F}{\gcd} (\theta - \theta') + a_F m \ell_E = \left( 1 - \dfrac{a_E N_E}{\gcd} \right) (\theta - \theta') + a_F m \ell_E \\
    &= \theta - \theta' - a_E m \ell_F + a_F m \ell_E
\end{align*}
\end{proof}
\end{remark}

In \cite{bobadilla2022}, Fernández de Bobadilla and Pe{\l}ka enriched A'Campo's construction to the symplectic category. More specifically, they constructed a smooth manifold with boundary that they called the \emph{A'Campo space} $A$, and a smooth map $f_A: A \to \C_\log$ such that the restriction $f_A|_{\partial A}: \partial A \to \partial \C_\log \cong \SS^1$ is a locally trivial fibration isomorphic to the Milnor fibration in the tube. Here $\C_\log \cong \R_{\geq 0} \times \SS^1$ is the real oriented blow-up of $\C$ at the origin. The space $A$ is endowed with a smooth $1$-form $\lambda_A$ that makes $f_A|_{f_A^{-1}(\partial \C_\log)}$ a Liouville fibration, see \cite[\S 5.2]{bobadilla2022}. Moreover, A'Campo's monodromy is an exact symplectomorphism with respect to this new structure. In \cite[\S 5.5.3]{bobadilla2022} it is explained how to grade this symplectomorphism following \cite[App. A]{mclean2019}, so altogether, one obtains a graded abstract contact open book $(\F,\lambda,\varphi)$. It turns out that it is graded isotopic to the graded abstract contact open book associated to the Milnor fibration in the tube, see \cite[\S 7.2]{bobadilla2022}.

The dynamical properties of $(\F, \lambda, \varphi)$ are the same as the dynamical properties of A'Campo's classical description of the monodromy which we have just recalled in the case of a plane curve. More specifically, there is a smooth map $\pi_A: A \to Y$ and the pieces in decomposition \eqref{fiber-decomposition} are realized as
\[
\F_E = \overline{\pi_A^{-1}(E^\circ)} \cap f_A^{-1}(0,0), \quad \F_{E,F} = \pi_A^{-1}(E \cap F) \cap f_A^{-1}(0,0).
\]

\begin{proposition}[{\cite[Proposition 7.4]{bobadilla2022}}] \label{undeformed-dynamics}
    Let $m \geq 1$ be an integer and let $\mu: (Y,\mathbf{E}) \to (\C^2,0)$ be an $m$-separating log resolution of $f$. Then, the fixed points of $\varphi^m$ decompose as
    \[
    \mathrm{Fix} (\varphi^m) = \bigsqcup_{\substack{E \in \mathcal{E} \cup \mathcal{S} \\ N_E \mid m}} \F_E.
    \]
    Every connected component $B$ of $\F_E$ is a codimension zero family of fixed points (see Definition \ref{def-czf}), with $\partial^- B = \varnothing$ and $\CZ(B) = 2m(\tfrac{\nu_E}{N_E} - 1)$.
\end{proposition}

\subsection{Deforming the monodromy} \label{deform-monodromy}
After endowing the A'Campo monodromy with a symplectic structure, Fernández de Bobadilla and Pe{\l}ka use the McLean spectral sequence to deduce some properties about its Floer homology. This is good enough for their purposes, but not for the explicit computation that we want to carry out. Indeed, it will follow from Theorem \ref{main-result} that the McLean spectral sequence associated to the graded abstract contact open book $(\F, \lambda, \varphi^m)$ does not degenerate at the first page in general. Our strategy is to deform $\varphi^m$ to obtain an isotopic abstract contact open book ---which will therefore have the same Floer homology--- whose associated McLean spectral sequence does degenerate at the first page.

Recall from Definition \ref{parts-resolution-graph} that, given a rupture divisor $R \in \mathcal{R}$, if there are any $m$-divisors in $\mathcal{F}_R$ then there is one which is closest to $R$, and we denote it by $E_\dom(m,R)$. Let $B_{m,R}$ be the piece of the Milnor fiber corresponding to divisors in $\mathcal{F}_R$ which are ``at least as far away'' from $R$ as $E_\dom(m,R)$, that is
\[
B_{m,R} \coloneqq \bigcup_{E \in \mathcal{F}_{m,R}} \F_E \cup \bigcup_{E,F \in \mathcal{F}_{m,R}} \F_{E,F}.
\]

We may describe the topology of $B_{m,R}$ in terms of Proposition \ref{acampo-pieces}. Indeed, if there are no $m$-divisors in $\mathcal{F}_R$, then $B_{m,R} = \varnothing$. If there are $m$-divisors in $\mathcal{F}_R$ but $R$ itself is not an $m$-divisor, then $E_\dom(m,R)$ has valency 2 and each connected component of $B_{m,R}$ is a finite chain of cylinders with a disk attached at the end, corresponding to the leaf in the end. In other words, $B_{m,R}$ is a disjoint union of $c(\mathcal{F}_R)$ closed disks (or $c(\mathcal{F}_{R,i})$ disks in the exceptional case where $R$ has two ends attached to it, where $\mathcal{F}_{R,i}$ is the end that contains $E_\dom(m,R)$). Finally, if $R$ is an $m$-divisor then $B_{m,R}$ is the result of capping $c(\mathcal{F}_R)$ boundary components of $\F_R$ with the disks coming from the divisors in the end (or 0 components if $R$ has no ends attached, or $c(\mathcal{F}_{R,1}) + c(\mathcal{F}_{R,2})$ components in the exceptional case). From there one easily computes the number of boundary components and the genus of each connected component from the analogous information about $\F_R$. This topological data is explicitly given below in Proposition \ref{deformed-dynamics}.

\label{deforming_discussion} The restriction of $\varphi^m$ to $B_{m,R}$ is isotopic to the identity. Indeed, if $B_{m,R} \neq \varnothing$ then the set $\F_{E_\dom(m,R)}$ is a neighborhood of the boundary of $B_{m,R}$ such that $\varphi^m|_{\F_{E_\dom(m,R)}} = \id$ and $B_{m,R} \setminus \F_{E_\dom(m,R)}$ is homeomorphic to a disjoint union of disks. Since the restriction of $\varphi^m$ to each of the connected components of $B_{m,R} \setminus \F_{E_\dom(m,R)}$ is a diffeomorphism of the 2-disk that fixes the boundary, it is smoothly isotopic to the identity by a theorem of Smale, see \cite{smale1959}. In fact Lemma \ref{exact-moser} below shows that the smooth isotopy between the restriction of $\varphi^m$ and the identity may be chosen to be an exact symplectomorphism for all $t$. Doing this for every $R$ we obtain a smooth isotopy $\psi_t: \F \to \F$ by compactly supported exact symplectomorphisms such that $\psi_0 = \varphi^m$, and ${\psi_1}_{|B_{m,R}}=\id$ for every $R \in \mathcal{R}$ such that $B_{m,R} \neq \varnothing$. We set $\widetilde{\varphi^m} \coloneqq \psi_1$.

The diffeomorphism $\varphi^m$ is an exact monodromy of the fibration obtained by pulling back $f_A: \partial A \to \mathbb{S}^1$ through the covering of degree $m$ of $\mathbb{S}^1$. In turn, $\widetilde{\varphi^m}$ is an exact monodromy for the same fibration because it is isotopic to $\varphi^m$. By \cite[Corollary 5.10]{bobadilla2022}, $\widetilde{\varphi^m}$ inherits a grading from the grading on the vertical symplectic frame bundle of $f_A$, and the abstract contact open books $(\F,\lambda,\varphi^m)$ and $(\F,\lambda,\widetilde{\varphi^m})$ are graded isotopic. Hence \cite[Proposition 6.2]{bobadilla2022} guarantees that the Floer homology of $(\F,\lambda,\widetilde{\varphi^m})$ remains unchanged, i.e.
\[
    \HF_\bullet(\varphi^m,+) \cong \HF_\bullet(\widetilde{\varphi^m},+).
\]
The symplectomorphism $\widetilde{\varphi^m}$ still has good dynamical properties as in Proposition \ref{undeformed-dynamics}:


\begin{proposition} \label{deformed-dynamics}
    The fixed points of $\widetilde{\varphi^m}$ decompose as
    \[
    \mathrm{Fix}(\widetilde{\varphi^m}) = \left( \bigsqcup_{\substack{R \in \mathcal{R}}} B_{m,R} \right) \sqcup \left( \bigsqcup_{R,S \in \mathcal{R} \cup \mathcal{S}} \bigsqcup_{\substack{E \in \mathcal{T}_{R,S} \\ N_E \mid m}} \F_E  \right) \sqcup \left( \bigsqcup_{j=1}^r \F_{\widetilde{C}^{[j]}}\right).
    \]
    The connected components of each $B_{m,R}$, $\F_E$ and $\F_{\widetilde{C}^{[j]}}$ are codimension zero families of fixed points (see Definition \ref{def-czf}), with Conley-Zehnder indexes 
        \begin{align*}
            \CZ(B_{m,R}) &= 2m\left(\frac{\nu_{E_\dom(m,R)}}{N_{E_\dom(m,R)}} - 1\right) \text{ for all } R \in \mathcal{R}, \\ 
            \CZ(\F_E) &= 2m\left(\frac{\nu_E}{N_E} - 1\right) \text{ for all } E \in \mathcal{E} \cup \mathcal{S}, \\
            \CZ(\F_{\widetilde{C}^{[j]}}) & = 0.
        \end{align*}
    The topology of these sets is the following:
        \begin{enumerate}[(i)]
            \item Each $\F_{\widetilde{C}^{[j]}}$ is a cylinder, i.e. a copy of $\SS^1 \times [0,1]$.

            \item \label{cylinders-dynamics} For each $E \in \mathcal{T}_{R,S}$ such that $N_E$ divides $m$, the set $\F_E$ is a disjoint union of $c(\mathcal{T}_{R,S})$ cylinders.

            \item \label{BmR-dynamics} For each $R \in \mathcal{R}$, if there are no $m$-divisors in $\mathcal{F}_R$ then $B_{m,R}$ is empty. If there are $m$-divisors in $\mathcal{F}_R$ but $N_R$ does not divide $m$, then $B_{m,R}$ is a disjoint union of $c(\mathcal{F}_R)$ closed disks (or $c(\mathcal{F}_{R,i})$ disks in the exceptional case where $R$ has two ends attached to it, where $\mathcal{F}_{R,i}$ is the end that contains $E_\dom(m,R)$). Finally, if $R$ is an $m$-divisor, then $B_{m,R}$ is the disjoint union of $c(R)$ compact orientable surfaces, each with $\tfrac{1}{c(R)}\sum_{S \in \mathcal{N}_R} c(\mathcal{T}_{R,S})$ boundary components and genus
            \[
            1+\frac{1}{2c(R)}\left(N_R(\mathrm{val}(R)-2)-\sum_{\substack{E \in (\mathcal{E} \cup \mathcal{S})\setminus \{R\} \\ R \cap E \neq \varnothing}} \gcd(N_R,N_E)\right).
            \]
        \end{enumerate}
\end{proposition}

\begin{proof}
    The decomposition of the fixed points follows from the construction. Items (i) and (ii) follow from \ref{undeformed-dynamics}, since the Conley-Zehnder index is locally constant in codimension zero families of fixed points and there are points in $\F_{E_\dom(m,R)} \subset B_{m,R}$ which have not been altered by the deformation. Item (iii) just collects the description of the $B_{m,R}$ that we made in the construction above.
\end{proof}

\begin{remark}
    The reader may have noticed the similarity between Theorem \ref{rupture-topology} and Proposition \ref{deformed-dynamics}, including the resemblance between $\mathfrak{Z}_{m,R}$ and $B_{m,R}$. In fact, this observation is at the core of the proof of Theorem \ref{main-result} in \S 5.
\end{remark}

We used the following lemma in the definition of $\widetilde{\varphi^m}$.

\begin{lemma} \label{exact-moser}
    Let $(M, \lambda)$ be a $2$-dimensional Liouville domain, and let $\varphi_0, \varphi_1: M \to M$ be compactly supported exact symplectomorphisms. If $\varphi_0,\varphi_1$ are smoothly isotopic by compactly supported orientation-preserving diffeomorphisms, then they are smoothly isotopic by compactly supported exact symplectomorphisms.
\end{lemma}
\begin{proof}
    Let $\varphi: M \times [0,1] \to M$ be a smooth isotopy from $\varphi_0$ to $\varphi_1$ such that $\varphi_s \coloneqq \varphi|_{M \times \{s\}}: M \to M$ is an orientation-preserving compactly supported diffeomorphism for every $s \in [0,1]$. By assumption, $\varphi_0$ and $\varphi_1$ are also exact symplectomorphisms.

    We wish to apply a parametrized version of Moser's trick, see \cite[\S 3.2]{mcduff-salamon2017}. Consider the smooth family of 1-forms $\lambda_{s,t} = (1-t)\lambda+t\varphi_s^\ast\lambda$, $s,t \in [0,1]$. The 2-form $d\lambda_{s,t}$ is a convex combination of $d\lambda$ and $\varphi_s^* d\lambda$, which define the same orientation at every point because $\varphi_s$ is orientation-preserving for every $s \in [0,1]$. Therefore $d\lambda_{s,t}$ is nonvanishing and closed. Since $M$ is 2-dimensional, this implies that $d\lambda_{s,t}$ is a symplectic form for every $s,t \in [0,1]$.

    Let $X_{s,t}$ be the unique vector field on $M$ such that
    \begin{equation} \label{equation-vector-field}
        \frac{d}{dt}\lambda_{s,t}+X_{s,t} \ \lrcorner \ d\lambda_{s,t} = 0.
    \end{equation}
    Note that $X_{s,t}$ depends smoothly on the parameters $s, t \in [0,1]$. We define a vector field $\overline{X}_t$ on $M \times [0,1]$ by setting $\overline{X}_t(p,s) \coloneqq (X_{s,t}(p),0)$. Clearly, $\overline{X}_t$ depends smoothly on $t \in [0,1]$, and we interpret it as a time-dependent vector field. Since $\varphi_s$ is compactly supported for every $s$, we have $\varphi_s|_{\partial M} = \id$ and therefore $\lambda_{s,t}$ is independent of $t$ on $\partial M$. Thus $\frac{d}{dt}\lambda_{s,t} = 0$ on $\partial M$ for every $s,t \in [0,1]$; and $X_{s,t} = 0$ on $\partial M$; and hence $\overline{X}_s = 0$ on $\partial M \times [0,1]$ too. Therefore the flow of the time-dependent vector field $\overline{X}_t$ exists, i.e. there is a one-parameter family of maps $\overline{\psi}_t:  M \times [0,1] \to M \times [0,1]$ satisfying
    \begin{align*}
        \frac{d}{dt}\overline{\psi}_t(p,s) & = \overline{X}_t(\overline{\psi}_t(p,s)), \ \ \text{and} \\
        \overline{\psi}_0(p,s) & = (p,s)
    \end{align*}
    for every $(p,s) \in M \times [0,1]$. Define $\psi_{s,t}: M \to M$ by the formula $\psi_{s,t}(p) \coloneqq \mathrm{pr}_M \circ \overline{\psi}_t(p,s)$, where $\mathrm{pr}_M: M \times [0,1] \to M$ is the natural projection on $M$. Therefore, $\psi_{s,t}$ satisfies
    \begin{align*}
        \frac{d}{dt}\psi_{s,t}(p,s) & = X_{s,t}(\psi_{s,t}(p,s)), \ \ \text{and} \\
        \psi_{s,0}(p) & = p
    \end{align*}
    for every $p \in M$, and $s,t \in [0,1]$, i.e. for fixed $s$, the flow of $X_{s,t}$ is $\psi_{s,t}$. Also, for fixed $s,t$, the vector field $X_{s,t}$ is zero on a neighborhood of $\partial M$, so the the map $\psi_{s,t}$ is compactly supported.
    
    Using Cartan's magic formula, we get
    \begin{align*}
        \frac{d}{dt}(\psi_{s,t}^\ast\lambda_{s,t}-\lambda) & = \psi_{s,t}^\ast\left(\frac{d}{dt}\lambda_{s,t}+\mathcal{L}_{X_{s,t}}\lambda_{s,t}\right) = \\
        & = \psi_{s,t}^\ast\left(\frac{d}{dt}\lambda_{s,t}+X_{s,t} \ \lrcorner \ d\lambda_{s,t}+d(X_{s,t} \ \lrcorner \ \lambda_{s,t})\right) \overset{(\ref{equation-vector-field})}{=} \\
        & =\psi_{s,t}^\ast\left(d(X_{s,t} \ \lrcorner \ \lambda_{s,t})\right) = d\psi_{s,t}^\ast\left(X_{s,t} \ \lrcorner \ \lambda_{s,t}\right),
    \end{align*}
    so $\frac{d}{dt}(\psi_{s,t}^\ast\lambda_{s,t}-\lambda)$ is exact for every $s,t \in [0,1]$. Moreover, $\psi_{s,0}^\ast\lambda_{s,0}-\lambda = \lambda_{s,0}-\lambda = \lambda-\lambda = 0$ for every $t \in [0,1]$. In other words, for every fixed $s \in [0,1]$ we have a path of 1-forms $\psi_{s,t}^\ast \lambda_{s,t}-\lambda$ (parametrized by $t$) that starts at an exact 1-form and whose derivative is exact. We conclude that
    \begin{equation} \label{exact}
        \parbox{11cm}{\centering {$\psi_{s,t}^\ast \lambda_{s,t}-\lambda$ is exact for every $s,t \in [0,1]$.}}
    \end{equation}

    Let $\varphi'_t \coloneqq \varphi_0 \circ \psi_{0,t}$. Then,
    \begin{align*}
        (\varphi'_t)^\ast\lambda-\lambda & = (\varphi_0 \circ \psi_{0,t})^\ast \lambda-\lambda = \psi_{0,t}^\ast\varphi_0^\ast \lambda-\lambda = \psi_{0,t}^\ast\lambda_{1,0}-\lambda = \\
        & = (\psi_{0,t}^\ast\lambda_{1,0}-\psi_{0,t}^\ast\lambda_{0,t})+(\psi_{0,t}^\ast\lambda_{0,t}-\lambda) = \\
        & = \psi_{0,t}^\ast(\lambda_{1,0}-\lambda_{0,t})+(\psi_{0,t}^\ast\lambda_{0,t}-\lambda).
    \end{align*}
    By (\ref{exact}), we know that the 1-form $\psi_{s,0}^\ast\lambda_{s,0}-\lambda$ is exact. In turn,
    \[
    \lambda_{1,0}-\lambda_{0,t} = \varphi_0^\ast\lambda-((1-t)\lambda+t\varphi_0^\ast\lambda) = (1-t)(\varphi_0^\ast\lambda-\lambda),
    \]
    which is exact because $\varphi_0$ is an exact symplectomorphism. Therefore, the maps $\{\varphi'_t\}_{t \in [0,1]}$ define a smooth isotopy by compactly supported exact symplectomorphisms from $\varphi_0$ to $\varphi_0 \circ \psi_{0,1}$. Analogously, the map $\varphi''_t \coloneqq \varphi_1 \circ \psi_{1,t}$ is a smooth isotopy by compactly supported exact symplectomorphisms from $\varphi_1$ to $\varphi_1 \circ \psi_{1,1}$.

    Finally the map $\varphi'''_s \coloneqq \varphi_s \circ \psi_{s,1}$ is an exact symplectomorphism for every $s \in [0,1]$. Indeed,
    \[
    ({\varphi_s'''})^\ast\lambda-\lambda = (\varphi_s \circ \psi_{s,1})^\ast\lambda -\lambda = \psi_{s,1}^\ast\varphi_s^\ast\lambda-\lambda = \psi_{s,1}^\ast\lambda_{s,1}-\lambda
    \]
    which is exact by (\ref{exact}). Thus, the maps $\{\varphi'''_s\}_{s \in [0,1]}$ define a smooth isotopy by compactly supported exact symplectomorphisms from $\varphi_0 \circ \psi_{0,1}$ to $\varphi_1 \circ \psi_{1,1}$.
    
    Concatenating the isotopies $\{\varphi'_t\}$, $\{\varphi''_t\}$ and $\{\varphi'''_s\}$, the claim is proved.
\end{proof}

\subsection{Degeneration of the McLean spectral sequence}
McLean \cite[App. C]{mclean2019} constructed a spectral sequence that converges to the Floer cohomology of a graded abstract contact open book with good dynamical properties. In this paper, we use a slightly more general version of it that appeared in \cite{bobadilla2022} and converges to Floer homology.

\begin{proposition}[{\cite[Proposition 6.3]{bobadilla2022}}] \label{spectral_seq}
Let $(M,\lambda,\varphi)$ be a graded abstract contact open book, $\mathcal{a}: M \to \R$ an action of $\varphi$ and $\dim M = 2n$. Assume that
\[
\Fix \varphi = \bigsqcup \B,
\]
where each $B \in \B$ is a codimension zero family of fixed points such that $\partial^-B = \varnothing$. Pick a map $\iota: \B \to \Z$ such that:
\begin{itemize}
    \item if $\mathcal{a}(B)=\mathcal{a}(B')$, then $\iota(B)=\iota(B')$,
    \item if $\mathcal{a}(B)<\mathcal{a}(B')$, then $\iota(B)<\iota(B').$
\end{itemize}
Then there is a spectral sequence
\[
E^1_{p,q}=\bigoplus_{\iota(B)=p} H_{n+p+q+\CZ_\varphi(B)}(B,\partial^+B) \Longrightarrow \HF_{p+q}(\varphi,+).
\]
\end{proposition}

\begin{remark}
    In \cite{bobadilla2022}, the Floer homology and the McLean spectral sequence are considered with coefficients in $\Z/2\Z$. However, their arguments work in exactly the same way if coefficients in $\Z$ are considered.
\end{remark}

Our goal is to study the degeneration properties of the McLean spectral sequence in the setting of the Milnor fiber of plane curve singularities. To that end, we adapt the arguments of \cite{seidel1996}. We start by giving a sufficient topological condition which implies the degeneration of the McLean spectral sequence at the first page.

\begin{lemma} \label{degeneration}
    Let $(M,\lambda,\varphi)$ and $\B$ be as in Proposition \ref{spectral_seq}. Assume that the following condition holds:
    \begin{equation} \label{condition}
        \parbox{14cm}{{If $\eta: [0,1] \to M$ is a path whose endpoints are fixed by $\varphi$, and $\varphi \circ \eta$ is homotopic to $\eta$ relative to the endpoints, then both endpoints of $\eta$ lie on the same $B \in \B$.}}
    \end{equation}
    Then the McLean spectral sequence degenerates at the first page.
\end{lemma}
\begin{proof}
    Let $\check{H}_{\delta,t}$ and $J_t$ be the time-dependent Hamiltonian and almost complex structure considered in \cite[\S 6.3.2]{bobadilla2022}. As discussed in \S 6.3.3 in loc. cit., every $\varphi$-twisted Hamiltonian loop is constant and equal to a point in some $B \in \B$. We claim that if there exists a Floer trajectory between two $\varphi$-twisted Hamiltonian loops, then the loops are points lying on the same $B \in \B$. Indeed, let $u: \R^2 \to M$ be a Floer trajectory between the constant loops $p_-$ and $p_+$. In particular, for every $(s,t) \in \R^2$ it satisfies
    \begin{itemize}
        \item $\varphi(u(s,t+1))=u(s,t)$,
        \item $\lim_{s \to \pm \infty} u(s,t) = p_{\pm}$.
    \end{itemize}
    From $u$ we can obtain a smooth homotopy $\bar{u}: [0,1] \times \R \to M$ relative to the endpoints $p_-$ and $p_+$ such that 
    $\varphi \circ \bar{u}(\cdot,1) = \bar{u}(\cdot,0)$, so by condition (\ref{condition}) there exists $B \in \B$ such that $p_-,p_+ \in B$, as desired.

    Recall the Floer complex $\mathrm{CF}_*(\check{H}_{\delta,t},J_t)$ is generated by the $\varphi$-twisted Hamiltonian loops, and graded by the minus Conley-Zehnder index. Its differential counts the signed number of Floer trajectories between $\varphi$-twisted Hamiltonian loops. The fact that there are no Floer trajectories between different families of $\B$ in particular implies that the complex $\mathrm{CF}_*(\check{H}_{\delta,t},J_t)$ can be expressed as a direct sum of complexes by grouping loops with the same action functional. Moreover, the filtration inducing the McLean spectral sequence 
    coincides with the natural filtration induced by the direct sum decomposition of $\mathrm{CF}_*(\check{H}_{\delta,t},J_t)$, so the spectral sequence degenerates at the first page.
\end{proof}

\begin{lemma} \label{topological-condition}
    The graded abstract contact open book $(\F,\lambda,\widetilde{\varphi^m})$ satisfies condition \eqref{condition} in Lemma \ref{degeneration}.
\end{lemma}
\begin{proof}
    Let $\eta$ be a path between two fixed points of $\dphi$ such that $\dphi \circ \eta$ is homotopic to $\eta$ relative to the endpoints. Let $N$ be a positive integer divisible by $N_E$ for every $E \in \mathcal{E}$. Then $(\dphi)^N$ is the identity on every $\F_E$ for $E \in \mathcal{S} \cup \mathcal{E}$, and also on every $\F_{E,F}$, if $E,F \in \mathcal{F}_{m,R}$ for some $R \in \mathcal{R}$ and $E \cap F \neq \varnothing$.

    Recall that a Dehn twist is isotopic to the identity if and only if the image of the fundamental group of its support is trivial in the fundamental group of the surface. Thus, the Dehn twists happening on the connected components of $M_{E,F}$ are isotopic to the identity if and only if $E, F \in \mathcal{F}_R$ for some $R \in \mathcal{R}$. Therefore, $(\dphi)^N$ is isotopic to a composition $T$ of Dehn twists which are not isotopic to the identity. In turn, $(\dphi)^N \circ \eta$ is homotopic to $\eta$, so by \cite[Lemma 3(ii)]{seidel1996} the endpoints of $\eta$ lie on the same connected component of $\Fix T$. Hence $\eta(0)$ and $\eta(1)$ lie on the same family of fixed points.
\end{proof}

\begin{proposition} \label{degeneration-braches}
    The McLean spectral sequence associated to the decomposition of Proposition \ref{deformed-dynamics} degenerates at the first page.
\end{proposition}
\begin{proof}
    It follows from combining Lemma \ref{degeneration} and Lemma \ref{topological-condition}.
\end{proof}

\section{Comparing (co)homologies}

Armed with the results of Sections \ref{section:contact-loci} and \ref{section:floer-homology}, we are now able to compute the cohomology of the contact loci and the Floer homology of the monodormy iterates associated to the plane curve $f$. Recall that we have decompositions
\begin{align*}
    \X_m & = \left( \bigsqcup_{R \in \mathcal{R}} \frakZ_{m,R} \right) \sqcup \left( \bigsqcup_{R,S \in \mathcal{R} \cup \mathcal{S}} \bigsqcup_{\substack{E \in \mathcal{T}_{R,S} \\ N_E \mid m}} \X_{m,E} \right) \\
    \mathrm{Fix}(\widetilde{\varphi^m}) & = \left( \bigsqcup_{R \in \mathcal{R}} B_{m,R} \right) \sqcup \left( \bigsqcup_{R,S \in \mathcal{R} \cup \mathcal{S}} \bigsqcup_{\substack{E \in \mathcal{T}_{R,S} \\ N_E \mid m}} \F_E  \right) \sqcup \left( \bigsqcup_{j=1}^r \F_{\widetilde{C}^{[j]}}\right)
\end{align*}

coming from Corollary \ref{contact-loci-decomposition} and Proposition \ref{deformed-dynamics} respectively.

\begin{proposition} \label{piecewise-isomorphisms}
    For every $m \geq 1$ there are isomorphisms
    \begin{align*}
        H^{\bullet + 4m-2m_{E_{\dom(m,R)}}\nu_{E_{\dom(m,R)}}}_c(\frakZ_{m,R}) &\cong H^\bullet(B_{m,R}, \partial B_{m,R}) \ &&\text{for } R \in \mathcal{R}, \\
        H^{\bullet + 4m-2m_E\nu_E}_c(\X_{m,E}) &\cong H^\bullet(\F_E, \partial \F_E) \ &&\text{for } E \in \mathcal{T}_{R,S} \text{ such that $N_E$ divides $m$}.
    \end{align*}    
\end{proposition}

\begin{proof}
    From Theorem \ref{XmE-computation}, we know that if $E$ is an exceptional divisor in the trunk $\mathcal{T}_{R,S}$, then $\X_{m,E}$ is homeomorphic to $\bigsqcup^{c(\mathcal{T}_{R,S})} \C^\times \times \C^{2m - m_E \nu_E}$. In turn, from Proposition \ref{deformed-dynamics}.\ref{cylinders-dynamics}, we have that $M_E$ is homeomorphic to a disjoint union of $c(\mathcal{T}_{R,S})$ cylinders. Hence we get the homeomorphism
    \[
    \X_{m,E} \cong (\F_E \setminus \partial \F_E) \times \C^{2m - m_E \nu_E}.
    \]
    Similarly, from Theorem \ref{rupture-topology} and Proposition \ref{deformed-dynamics}.\ref{BmR-dynamics}, we obtain the homeomorphism
    \[
    \frakZ_{m,R} \cong (B_{m,R} \setminus \partial B_{m,R}) \times \C^{2m-m_{E_\dom(m,R)}\nu_{E_\dom(m,R)}}.
    \]
    Thus the result follows from the fact that compactly supported cohomology is canonically isomorphic to cohomology of a compactification relative to the boundary, and dimension counting.
    
\end{proof}

Note that the shift in the grading of the cohomology groups in Proposition \ref{piecewise-isomorphisms} depends on the divisors $E$ and $R$. This dependence on the divisor cancels out the shift in grading in the McLean spectral sequence.

\begin{lemma} \label{shift-computation}
    We have
    \begin{align*}
        1 + \CZ(B_{m,R}) + 2m \left(2 - \frac{\nu_{E_\dom(m,R)}}{N_{E_\dom(m,R)}}\right) &= 2m + 1 &&\text{for } R \in \mathcal{R}, \\ 
        1 + \CZ(\F_E) + 2m \left(2 - \frac{\nu_E}{N_E}\right) &= 2m + 1 &&\text{for } E \in \mathcal{E} \cup \mathcal{S}.
    \end{align*}
\end{lemma}

\begin{proof}
    The result follows immediately from Proposition \ref{deformed-dynamics}.
\end{proof}

Putting everything together, we can prove our main result.

\begin{proof}[Proof of Theorem \ref{main-result}]
    There are isomorphisms
    \begin{align} 
    H^\bullet_c(\X_m) &\cong \left(\bigoplus_{R \in \mathcal{R}} H^\bullet_c(\frakZ_{m,R})\right) \oplus \left(\bigoplus_{R, S \in \mathcal{R} \cup \mathcal{S}} \bigoplus_{\substack{E \in \mathcal{T}_{R,S} \\ N_E \mid m}} H^\bullet_c(\X_{m,E})\right) \label{cohomologies} \\
    \HF_\bullet(\varphi^m,+) &\cong \left(\bigoplus_{R \in \mathcal{R}} H_{\bullet+1+\CZ(B_{m,R})}(B_{m,R}, \partial B_{m,R})\right) \oplus \nonumber \\
    &\phantom{\cong} \oplus \left(\bigoplus_{R, S \in \mathcal{R} \cup \mathcal{S}} \bigoplus_{\substack{E \in \mathcal{T}_{R,S} \\ N_E \mid m}} H_{\bullet+1+\CZ(\F_E)}(\F_E, \partial \F_E)\right). \nonumber
    \end{align}
    The first isomorphism follows immediately from Corollary \ref{contact-loci-decomposition}, while the second isomorphism follows from the degeneration of the McLean spectral sequence in Proposition \ref{degeneration-braches}. An important observation is that the pieces $\F_{\widetilde{C}^{[j]}}$ do not contribute to the Floer homology. Indeed, in this case the piece $\F_{\widetilde{C}^{[j]}}$ is a cylinder, and the subset $\partial^+ \F_{\widetilde{C}^{[j]}}$ is one its boundary components. Then the claim follows because the homology of a cylinder relative to one of its boundary components is zero. Now Proposition \ref{piecewise-isomorphisms} and Lemma \ref{shift-computation} give the result.
\end{proof}

\begin{remark}
    Note that, since the spectral sequence degenerates on the first page, we did not need to compute the value of an action of $\varphi$ at the sets $B_{m,R}$ and $\F_E$. 
    Nevertheless, by \cite[Proposition 7.4]{bobadilla2022}, the action can be chosen so that its values are
    \begin{align*}
        \mathcal{a}(B_{m,R}) & = (t_0b \cdot b_{E_{\mathrm{dom}}} + \varepsilon_0)\frac{m}{N_{E_{\mathrm{dom}}}}, \\
        \mathcal{a}(M_E) & = (t_0b \cdot b_E + \varepsilon_0)\frac{m}{N_E}, \\
        \mathcal{a}(M_{\widetilde{C}^{[j]}}) & = \varepsilon_0 m.
    \end{align*}
    Here, $H = \sum_{E \in \mathcal{E}} b_E E$ is an 
    ample divisor that is fixed during the construction of the A'Campo abstract contact open book with $b_E < 0$ for every $E \in \mathcal{E}$, $b > 0$ an integer such that $bH$ is very ample, and $0 < \varepsilon_0 \ll t_0 \ll 1$ are arbitrarily small positive numbers as in \cite[\S 7.2]{bobadilla2022}.
\end{remark}

\begin{remark}
    The isomorphism of Theorem \ref{main-result} is not canonical because it involves the isomorphism of a vector space with its dual, but also because we used a spectral sequence to compute the Floer homology, so we can only know the associated graded of $\HF_\bullet(\varphi^m, +)$, and in particular the dimension of the vector space at each degree. However, there is a work in progress by M. McLean in which a canonical isomorphism is proposed, see the slides of his talk in: \url{https://www.math.stonybrook.edu/~markmclean/talks/FloerArcTalk.pdf}. 
\end{remark}

We collect all of our computations into the following result.

\begin{corollary}
    The compactly supported cohomology of $\X_m$ is given by
    \begin{align*}
        H^\bullet_c(\X_m) \cong & \left(\bigoplus_{\substack{R \in \mathcal{R} \text{ such that} \\ N_R \mid m}} \bigoplus^{c(R)} H_c^{\bullet
        }(T_{m,R} \times \C^{2m - m_R \nu_R})\right) \oplus \\
        \oplus & \left(\bigoplus_{\substack{R \in \mathcal{R} \text{ such that} \\ N_R \nmid m \text{ and} \\
        \exists E \in \mathcal{F}_R \text{ s.t. } N_E \mid m}} \bigoplus^{c(\mathcal{F}_R)} H_c^{\bullet
        }(\C \times \C^{2m - m_{E_\dom(m,R)} \nu_{E_\dom(m,R)}})\right) \oplus \\
        \oplus & \left(\bigoplus_{R, S \in \mathcal{R} \cup \mathcal{S}} \bigoplus_{\substack{E \in \mathcal{T}_{R,S} \\ \text{such that } N_E \mid m}} \bigoplus^{c(\mathcal{T}_{R,S})} H_c^{\bullet
        }(\C^\times \times \C^{2m - m_E \nu_E})\right).
    \end{align*}
    Here, for every $m$-divisor $R \in \mathcal{R}$, the surface $T_{m,R}$ is a compact Riemann surface of genus
    \[
    1+\frac{1}{2c(R)} \Bigg( N_R(\mathrm{val}(R)-2)-\sum_{\substack{E \in (\mathcal{E} \cup \mathcal{S})\setminus \{R\} \\ R \cap E \neq \varnothing}} \gcd(N_R,N_E) \Bigg)
    \]
    with
    $
    \frac{1}{c(R)}\sum_{S \in \mathcal{N}_R} \gcd(N_R,N_S)
    $
    points removed. The end $\mathcal{F}_R$ should be interpreted to mean $\mathcal{F}_{R,i}$ in the exceptional case explained in Lemma \ref{end-well-defined}. Also, $\mathrm{HF}_\bullet(\varphi^m,+) \cong H^{\bullet+2m+1}_c(\X_m)$.
\end{corollary}
\begin{proof}
    The computation of $H^\bullet_c(\X_m)$ follows from combining equation \eqref{cohomologies} and Theorems \ref{XmE-computation} and \ref{rupture-topology}. The isomorphism with Floer homology is 
    Theorem \ref{main-result}.
\end{proof}

\begin{corollary}
    Both the homeomorphism class of the restricted $m$-contact locus and the Floer homology of the $m$-th iterate of a symplectic monodromy are invariants of the embedded topological type of the plane curve $(C,0) \subset (\C^2,0)$.
\end{corollary}

\bibliographystyle{amsplain}
\bibliography{plane-curves}

\vspace{2cm}
\noindent
\textsc{Javier de la Bodega}. \url{jdelabodega@bcamath.org}, \url{javier.delabodega@kuleuven.be} \\
Basque Center for Applied Mathematics. Alameda Mazarredo 14, 48009 Bilbao, Spain. \\
KU Leuven, Department of Mathematics. Celestijnenlaan 200B, 3001 Heverlee, Belgium. \\

\noindent
\textsc{Eduardo de Lorenzo Poza}. \url{eduardo.delorenzopoza@kuleuven.be} \\
Basque Center for Applied Mathematics. Alameda Mazarredo 14, 48009 Bilbao, Spain. \\
KU Leuven, Department of Mathematics. Celestijnenlaan 200B, 3001 Heverlee, Belgium.

\end{document}